\documentclass{article}
\usepackage{marsden_article2}
\usepackage{amssymb}
\usepackage{amscd}
\usepackage{graphicx}
\usepackage{amsmath}
\usepackage{moreverb}
\usepackage{epstopdf}
\usepackage{indentfirst}
\usepackage{array}
\usepackage{tikz}
\usepackage{setspace}

\newtheorem{example}[theorem]{Example}

\begin{document}

\author{{Sina Ober-Bl\"obaum\thanks{Corresponding author. Email: sinaob@math.uni-paderborn.de} and Nils Saake\thanks{Email: snils@mail.uni-paderborn.de}}\\[2pt]
{\small Computational Dynamics and Optimal Control, Department of Mathematics}, \\ {\small University of Paderborn, Warburger Str.~100, 33098 Paderborn, Germany}}
\title{Construction and analysis of higher order Galerkin variational integrators}

\maketitle

\begin{abstract}
{In this work we derive and analyze variational integrators of higher order for the structure-preserving simulation of mechanical systems. The construction is based on a space of polynomials together with Gauss and Lobatto quadrature rules to approximate the relevant integrals in the variational principle. The use of higher order schemes increases the accuracy of the discrete solution and thereby decrease the computational cost while the preservation properties of the scheme are still guaranteed. 
The order of convergence of the resulting variational integrators are investigated numerically and it is discussed which combination of space of polynomials and quadrature rules provide optimal convergence rates. For particular integrators the order can be increased compared to the Galerkin variational integrators previously introduced in \cite{MaWe01}.
Furthermore, linear stability properties, time reversibility, structure-preserving properties as well as efficiency for the constructed variational integrators are investigated and demonstrated by numerical examples.}
{discrete variational mechanics; numerical convergence analysis; symplectic methods; variational integrators}
\end{abstract}

\section{Introduction}\label{sec:intro}

During the last years the development of geometric numerical integrators has been of high interest in numerical integration theory. 
Geometric integrators are structure-peserving integrators with the goal to capture the dynamical system's behavior in a most realistic way (\cite{MaWe01,HaLuWa,Reic94}). Using structure-preserving methods for the simulation of mechanical systems, specific properties of the underlying system are handed down to the numerical solution, for example, the energy of a conservative system shows no numerical drift or the first integrals induced by symmetries are preserved exactly.
One particular class of structure-preserving integrators is the class of variational integrators, introduced in \cite{MaWe01,Sur90}, and which has been further developed and extended to different systems and applications during the last years.
Variational integrators (\cite{MaWe01}) are based on a discrete variational formulation of the underlying system, e.g.~based on a discrete version of Hamilton's principle for conservative mechanical systems.
The resulting integrators are symplectic and momentum-preserving and have an excellent long-time energy behavior.

By choosing different variational formulations (e.g.~Hamilton, Lagrange-d'Alembert, Hamilton-Pontryagin, etc.), variational integrators have been developed for a large class of problems: These involve classical conservative mechanical systems (for an overview see \cite{LMOW04,LMOWe04}), forced and controlled systems (\cite{ObJuMa10,Kane00}) , constrained systems (holonomic (\cite{leyendecker07-2, DMOCC}) and nonholonomic systems (\cite{KoMa2010, CoMa01})), nonsmooth systems (\cite{FMOW03}), stochastic systems (\cite{BRO08}), multiscale systems (\cite{TaOwMa2010, LO12, StGr09}) and Lagrangian PDE systems (\cite{Lew03, MPS98}). The applicability of variational integrators is not restricted to mechanical systems. In \cite{OTCOM13} variational integrators have been developed for the structure-preserving simulation of electric circuits.

Of special interest is the construction of higher order symplectic integrators:
To ensure moderate computational costs for long-time simulations, typically first or second order integrators are used. However, many applications, in particular problems in space mission design, demand more accurate discretization schemes.
There are mainly three different ways of constructing symplectic integrators of higher order (cf.~\cite{Leimkuhler2004}): (i) By applying  composition methods higher order symplectic schemes can be constructed in a systematic way based on a splitting of the Hamiltonian into explicitly solvable subproblems (for an overview see \cite{McQu02,Yoshida1990262}).
(ii) For (partitioned) Runge-Kutta methods there is a well-developed order theory which can be used to identify higher order symplectic Runge-Kutta methods (order conditions on the coefficients have first been introduced by \cite{Sa88,La88,Sur89,Sun93}). However, the identification of the coefficients for a symplectic Runge-Kutta scheme of desired order is not trivial and quite involved.
(iii) In contrast, generating functions (see e.g.~\cite{arnold89}) can be constructed that automatically guarantees the symplecticity of the associated numerical method. 
Since variational integrators rely on the approximation of the action, a generating function of first kind (cf.~\cite{MaWe01}), 
we focus on the latter approach for the construction of higher order methods.

Of particular interest are \emph{Galerkin variational integrators} which have been already studied in e.g.~\cite{MaWe01,Leok2011,HaLe13}. They rely on the approximation of the action based on a choice of a finite-dimensional function space and a numerical quadrature formula.
In \cite{Leok2011} Galerkin and shooting-based constructions for discrete Lagrangian are presented. Rather than choosing a infinite-dimensional function space, the shooting-based construction depends on a choice of a numerical quadrature formula together with a one-step method. 
In \cite{HaLe13} a convergence analysis for Galerkin type variational integrators is established showing that under suitable assumptions the integrators inherit the order of convergence given by the finite-dimensional approximation space used in the construction.
More detailed, it is shown that a Galerkin variational integrator based on an $s+1$-dimensional function space (e.g.~the space of polynomials of degree $s$) and a quadrature rule of order $s$ has convergence oder $s$.

In this contribution we numerically demonstrate that the convergence order of the variational integrator can even be increased if higher order quadrature rules are used.
We focus on a particular class of Galerkin variational integrators: As finite-dimensional function space we choose the space of polynomials of degree $s$. Furthermore, as quadrature rules we focus on the Gauss and Lobatto quadrature formula. However, in contrast to \cite{MaWe01} we do not restrict the number $r$ of quadrature points being equal to the polynomial order $s$. 
For two numerical examples we investigate which combination of space of polynomials, quadrature rules and number of quadrature points provide optimal convergence rates.\footnote{The generalization $r\not=s$ is also described in \cite{HaLe13}, however the influence of the relation of the number of quadrature points and the polynomial degree is not investigated. }
In particular, the numerical results indicate that the order of the higher order variational integrator constructed by a polynomial of degree $s$ and a quadrature rule of order $u$ is $\min{(2s,u)}$. Thus, the integrator order can be increased to $2s$ for sufficient accurate quadrature rules. While the focus on this work lies on numerical investigations, a formal proof of this superconvergence result is still subject of ongoing research.
Based on the numerical results, we perform a numerical analysis regarding efficiency versus accuracy (see also \cite{Sa12}).
Furthermore, we investigate analytically and numerically preservation properties, time reversibility and linear stability of the constructed integrators. Whereas preservation properties and time reversibly has also been subject of previous works for particular Galerkin variational integrates (see e.g.~\cite{Leok2011}), the stability analysis provides another new contribution. 

\paragraph{Outline}

In Section~\ref{sec:varmech} we recall the basic definitions and concepts of variational mechanics and variational integrators.
In Section~\ref{sec:highvarint} the higher order integrators are constructed following the Galerkin approach introduced in \cite{MaWe01}. Properties of the Galerkin variational integrators are presented in Section~\ref{sec:prop}. In Section~\ref{subsec:prop} preservation properties, such as symplecticity and preservation of momentum maps are discussed. In Section~\ref{subsec:timerevers} it is shown under which conditions the constructed integrators are time-reversible. Furthermore, in Section~\ref{subsec:linstab} a linear stability analysis is performed for specific examples showing in which region the constructed higher order variational integrators are asymptotically stable (in the sense that the growth of the solution is asymptotically bounded (cf.~\cite{Leimkuhler2004}). A-stability for a particular class of variational integrators is shown.
In Section~\ref{sec:numex} the numerical convergence analysis by means of two numerical examples, the harmonic oscillator and the Kepler problem, is performed using different combinations of the polynomial degree $s$ and the number $r$ of quadrature points. Furthermore, the relation of computational efficiency and accuracy for two different classes of variational integrators is investigated. Finally, we conclude with a summary of the results and an outlook for future work in Section~\ref{sec:concl}.

\section{Variational mechanics}\label{sec:varmech}

\subsection{Hamilton's principle and Euler-Lagrange equations}\label{subsec:Hamprinc}

Consider a mechanical system defined on the $n$-dimensional configuration manifold $Q$ with corresponding tangent bundle $TQ$ and cotangent bundle $T^*Q$. Let $q(t)\in Q$ and $\dot{q}(t)\in T_{q(t)}Q$, $t\in [0,T]$ be the time-dependent configuration and velocity of the system.

The \emph{action} $\mathfrak{S}:C^2([0,T],Q) \rightarrow \mathbb{R}$ of a mechanical system is defined as the time integral of the Lagrangian, i.e.,
\begin{equation}\label{eq:action}
\mathfrak{S}(q) = \int_0^T {L}({q}(t),\dot{{q}}(t))\, dt
\end{equation}
where the $C^2$-Lagrangian ${L}:TQ \rightarrow \mathbb{R}$ consists of kinetic minus potential energy. \emph{Hamilton's principle} seeks curves $q \in C^2([0,T],Q)$ with fixed initial value $q(0)$ and fixed final value $q(T)$ satisfying
\begin{equation}\label{eq:varprinc}
\delta \mathfrak{S}(q)=0
\end{equation}
for all variations $\delta q \in T_qC^2([0,T],Q)$.
This leads to the \emph{Euler-Lagrange equations}
\begin{equation}\label{eq:el}
\frac{d}{dt} \frac{\partial L}{\partial \dot{{q}}} - \frac{\partial L}{\partial {{q}}} = 0
\end{equation}
which are second-order differential equations describing the dynamics for conservative systems.

\subsection{Discrete Hamilton's principle and discrete Euler-Lagrange equations}\label{subsec:discreteHam}

The concept of variational integrators is based on a discretization of the variational principle~\eqref{eq:varprinc}.
Consider a time grid $\Delta t = \{t_k = kh\, | \, k=0,\ldots,N\}$, $Nh=T$, where $N$ is a positive integer and $h$ the step size. We replace the configuration ${q}(t)$ by a discrete curve ${q}_d= \{ {q}_k \}_{k=0}^N$ with ${q}_k={q}_d(t_k)$ as approximations to ${q}(t_k)$. We define a discrete Lagrangian $L_d: Q\times Q\rightarrow \mathbb{R}$
\begin{equation}\label{eq:approxaction} 
{L}_d({q}_k,{q}_{k+1}) \approx \int_{t_k}^{t_{k+1}}{L}({q}(t),{\dot{q}}(t))\, dt
\end{equation}
that approximates the action on $[t_k,t_{k+1}]$ based on two neighboring discrete configurations ${q}_k$ and ${q}_{k+1}$. 
The \emph{discrete action} $\mathfrak{S}_d:Q^{N+1}\rightarrow\mathbb{R}$ is defined as 
\[
\mathfrak{S}_d(q_d) = \sum_{k=0}^{N-1} L_d(q_k,q_{k+1}).
\]
The \emph{discrete Hamilton principle} is formulated by finding stationary points of the discrete action given by
\begin{equation}\label{eq:dlda}
\delta  \mathfrak{S}_d(q_d)=0
\end{equation}
with $\delta {q}_0 = \delta {q}_N = 0$. This gives the \emph{discrete Euler-Lagrange equations (DEL)}
\begin{equation}\label{eq:discEL}
D_1{L}_d({q}_k,{q}_{k+1}) +  D_2 {L}_d({q}_{k-1},{q}_{k}) =0
\end{equation}
for $k=1,\ldots,N-1$ and with $D_i$ being the derivative w.r.t.~the $i$-th argument.
Equation \eqref{eq:discEL} provides a discrete iteration scheme for \eqref{eq:el} that determines ${q}_{k+1}$ for given ${q}_{k-1}$ and ${q}_k$. It is also known as the \emph{discrete Lagrangian map} $F_{L_d}:Q\times Q \rightarrow Q\times Q$, given by $F_{L_d}(q_{k-1},q_k)=(q_k,q_{k+1})$ and $(q_{k-1},q_k), (q_k,q_{k+1})$ satisfy \eqref{eq:discEL}.
The discrete iteration schemes derived by a discrete variational principle are called \emph{variational integrators} and are well-known to be symplectic and momentum-preserving and exhibit excellent long-time energy behavior (cf.~Section~\ref{subsec:prop}). 

 The \emph{discrete Legendre transforms} $\mathbb{F}^\pm L_d: Q\times Q\rightarrow T^*Q$ provide discrete expressions for the conjugate momenta by 
\begin{align*}
&\mathbb{F}^- L_d:  (q_k,q_{k+1}) \rightarrow (q_k, p^-_k) = (q_k,-D_1 L({q}_k,{q}_{k+1}))\quad\text{and}\quad\\
&\mathbb{F}^+ L_d:  (q_{k-1},q_{k}) \rightarrow (q_{k},p^+_{k}) = (q_k, D_2 L({q}_{k-1},{q}_{k})). 
\end{align*}
Note that \eqref{eq:discEL} can be equivalently written as $p^-_k =p^+_k$ and the \emph{discrete Hamiltonian map} $\tilde{F}_{L_d}:T^*Q \rightarrow T^*Q$ defined by
\[
\tilde{F}_{L_{d}}: (q_k,p_k) \rightarrow (q_{k+1},p_{k+1})=\mathbb{F}^{\pm}L_{d}\circ F_{L{}_{d}}\circ(\mathbb{F}^{\pm}L_{d})^{-1}(q_k,p_k).
\]
is equivalent to the discrete Lagrangian map.


\section{Higher order Variational Integrators}\label{sec:highvarint}

\subsection{Approximation of the action integral}\label{subsec:approxint}

The approximation of the action integral consists of two approximation steps: the approximation of the space of trajectories and the approximation of the integral of the Lagrangian by appropriate quadrature rules.

We approximate the space of trajectories $\mathcal{C}([0,h],Q) = \{ q:[0,h] \rightarrow Q\; | \; q(0) = q_a, q(h) = q_b \}$ by a finite-dimensional approximation $\mathcal{C}^s([0,h],Q) \subset \mathcal{C}([0,h],Q)$ of the trajectory space given by
\[
\mathcal{C}^s([0,h],Q) = \{ q \in \mathcal{C}([0,h],Q) \; | \; q \in \Pi^s \},
\]
with $\Pi^s$ being the space of polynomials of degree $s$. 
Given $s+1$ control points $0 = d_0<d_1<\cdots<d_{s-1}<d_s=1$ and $s+1$ configurations ${q}_0 = (q_0^0,q_0^1,q_0^2,\dots,q_0^{s-1},q_0^s)\in Q^{s+1}$ with $q_0^0=q_a$ and $q_0^s=q_b$,  the degree $s$ polynomial $q_d(t;{q}_0,h)$ which passes through each $q_0^\nu$ at time $d_\nu h$, that is, $q_d(d_\nu h)=q_0^\nu$ for $\nu =0,\dots,s$, is uniquely defined. 

With the Lagrange polynomial $l_{\nu,s}: [0,1] \rightarrow \mathbb{R}$
\[
l_{\nu,s}(\tau) = \prod_{0\le i\le s, i\ne \nu} \frac{\tau-d_i}{d_\nu-d_i}
\]
we obtain $q_d(t;{q}_0,h)$ with $t\in [0,h]$ as
\[
q_d(t;{q}_0,h) = \sum_{\nu=0}^s q_0^\nu l_{\nu,s}\left(\frac{t}{h}\right).
\]
With 
\[
l_{\nu,s}(d_{i})=\begin{cases}
1 & ,i=\nu\\
0 & ,i\neq \nu
\end{cases}
\]
for all $i=0,\ldots,s$, we have $q_{d}(d_{i}h;{q}_{0},h)=q_{0}^{i}$. The derivative of $q_d(t;{q}_0,h)$ w.r.t.~$t$ provides an approximation of $\dot{q}$ on $[0,h]$ as
\[
\dot{q}_{d}(t;{q}_{0},h)=\frac{1}{h}\sum_{\nu=0}^{s}q_{0}^{\nu}\dot{l}_{\nu,s}\left(\frac{t}{h}\right).
\]
To approximate the trajectory $q:[0,T]\rightarrow Q$, we divide the time interval $[0,T]$ in $N=T/h$ sub intervals of length $h$ as
\[
[0,T]=\bigcup_{k=0}^{N-1}[kh,(k+1)h].
\]
On all sub intervals we approximate $q:[0,T]\rightarrow Q$ piecewise by the polynomials $q_{d,k}:[kh,(k+1)h]\rightarrow Q$ defined by
\[
q_{d,k}(t;{q}_{k},h):=q_{d}(t-kh;{q}_{k},h)\quad t\in[kh,(k+1)h]
\]
with ${q}_{k}=(q_{k}^{0},q_{k}^{1},...,q_{k}^{s-1},q_{k}^{s})$ and $k=0,\ldots,N-1$. To obtain a continuous approximation on $[0,T]$, we set $q_{k}^{s}=q_{k+1}^{0}$ for all $k=0,\ldots,N-2$.

For the approximation of the action integral \eqref{eq:action} we replace the curves $q(t)$ and $\dot{q}(t)$ by the piecewise polynomials $q_{d,k}(t;{q}_{k},h)$ and $\dot{q}_{d,k}(t;{q}_{k},h)$, $k=0,\ldots,N$, and approximate
\begin{equation}\label{eq:action_qd}
\int_{kh}^{(k+1)h} L(q_{d,k}(t;{q}_{k},h)  , \dot{q}_{d,k}(t;{q}_{k},h)) \, dt
\end{equation}
on each time interval $[kh,(k+1)h],\, k=0,\ldots,{N-1}$, by choosing a numerical quadrature rule $(b_i,c_i)_{i=1}^r$ w.r.t.~the time interval $[0,1]$ with quadrature points $c_i \in [0,1]$ and weights $b_i, \; i=1,\dots, r$.
The choice of quadrature rule should be adapted to the desired order of accuracy of the integrator since the order of the quadrature rule provides an upper bound for the order of the variational integrator (cf.~e.g.~\cite{Leok2011}).
By applying the quadrature rule $(b_i,c_i)_{i=1}^r$ to the integral \eqref{eq:action_qd} we define the discrete Lagrangian $L_{d,k}$ as 
\begin{eqnarray*}
L_{d,k}=L_{d}({q}_{k}=(q_{k}^{0},...,q_{k}^{s}),h) & = & h\sum_{i=1}^{r}b_{i}L(q_{d,k}(c_{i}h+kh;{q}_{k},h),\dot{q}_{d,k}(c_{i}h+kh;{q}_{k},h))\\
 & = & h\sum_{i=1}^{r}b_{i}L(q_{d}(c_{i}h;{q}_{k},h),\dot{q}_{d}(c_{i}h;{q}_{k},h))
\end{eqnarray*}
which provides an approximation of the action on the interval $[kh, (k+1)h]$ as $L_{d,k}  \approx  \int_{kh}^{(k+1)h}L(q(t),\dot{q}(t))\,dt$.
Note that the discrete Lagrangian depends on $s+1$ configurations ${q}_k$ and the step size $h$. In the following, we write $L_d({q}_k)$ for $L_d({q}_k,h)$.
Finally, we define the discrete action sum over the entire trajectory to be
\begin{equation}\label{eq:discreteaction}
\mathfrak{S}_{d}({q}_{0},\ldots,{q}_{N-1})=\sum_{k=0}^{N-1}L_{d}({q}_{k})
\end{equation}
which is an approximation of the action sum on $[0,T]$ as $\mathfrak{S}_{d}({q}_{0},\ldots,{q}_{N-1}) \approx \mathfrak{S}(q)$.

\subsection{Discrete Hamilton's principle}\label{sub:discrteHam}
For the discrete action $\mathfrak{S}_{d}$ defined in \eqref{eq:discreteaction} we can apply discrete Hamilton's principle as described in Section~\ref{subsec:discreteHam}. Since we want to determine discrete approximations of curves for which the discrete action is stationary, the derivatives of the action w.r.t.~$q_k^\nu$ have to vanish for all $k=0,\ldots,N-1$ and $\nu=0,\ldots,s$.
This leads for $k=0,\ldots,N-1$ and $\nu = 1,\ldots,s-1$ to
\begin{eqnarray*}
0 & = & \frac{\partial \mathfrak{S}_{d}}{\partial q_{k}^{\nu}}({q}_{0},...,{q}_{N-1}) =  \frac{\partial L_{d}}{\partial q_{k}^{\nu}}({q}_{k})\\
 & = & h\sum_{i=1}^{r}b_{i}\left(\frac{\partial L}{\partial q}(c_{i}h;{q}_{k})\frac{\partial q_{d}}{\partial q_{k}^{\nu}}+\frac{\partial L}{\partial\dot{q}}(c_{i}h;{q}_{k})\frac{\partial \dot{q}_{d}}{\partial q_{k}^{\nu}}\right)\\
 & = & h\sum_{i=1}^{r}b_{i}\left(\frac{\partial L}{\partial q}(c_{i}h;{q}_{k})l_{\nu,s}(c_{i})+\frac{\partial L}{\partial\dot{q}}(c_{i}h;q_{k})\frac{1}{h}\dot{l}_{\nu,s}(c_{i})\right).
\end{eqnarray*}
Note that we use the short notation $\frac{\partial L}{\partial q}(c_{i}h;{q}_{k})=\frac{\partial L}{\partial q}(q_{d}(c_{i}h;{q}_{k},h),\dot{q}_{d}(c_{i}h;{q}_{k},h))$ and that we have $\frac{\partial q_{d}}{\partial q_{k}^{\nu}}=\frac{\partial q_{d}(c_{i}h;{q}_{k},h)}{\partial q_{k}^{\nu}}=l_{\nu,s}(c_{i})$. The analog holds for the other two terms.
With $ q_{k-1}^{s}=q_{k}^{0}$ for all $k=1,\ldots,N-1$ we obtain for $\nu=0$ and $\nu = s$ 

\begin{eqnarray*}
0 & = &\frac{\partial \mathfrak{S}_{d}}{\partial q_{k-1}^{s}}({q}_{0},...,{q}_{N-1}) = \frac{\partial \mathfrak{S}_{d}}{\partial q_{k}^{0}}({q}_{0},...,{q}_{N-1}) = \frac{\partial L_{d}({q}_{k-1})}{\partial q_{k-1}^{s}}+\frac{\partial L_{d}({q}_{k})}{\partial q_{k}^{0}}\\
 & = & h\sum_{i=1}^{r}b_{i}\left(\frac{\partial L}{\partial q}(c_{i}h;{q}_{k-1})l_{s,s}(c_{i})+\frac{\partial L}{\partial\dot{q}}(c_{i}h;{q}_{k-1})\frac{1}{h}\dot{l}_{s,s}(c_{i})\right)\\
 &  & +h\sum_{i=1}^{r}b_{i}\left(\frac{\partial L}{\partial q}(c_{i}h;{q}_{k})l_{0,s}(c_{i})+\frac{\partial L}{\partial\dot{q}}(c_{i}h;{q}_{k})\frac{1}{h}\dot{l}_{0,s}(c_{i})\right).
\end{eqnarray*}

With the notation $D_{i}L_{d}(q_{k}^{0},\ldots,q_{k}^{s}):=\frac{\partial L_{d}(q_{k})}{\partial q_{k}^{i-1}}$ we obtain the \emph{discrete Euler-Lagrange equations}

\begin{equation}\label{eq:del_1}
D_{s+1}L_{d}(q_{k-1}^{0},\ldots,q_{k-1}^{s})+D_{1}L_{d}(q_{k}^{0},\ldots,q_{k}^{s})=0,
\end{equation}
\begin{equation}\label{eq:del_2}
D_{i}L_{d}(q_{k}^{0},\ldots,q_{k}^{s})=0\quad \forall i=2,\ldots,s,
\end{equation}
for $k=1,\ldots,N-1$. A sequence $\{{q}_{k}\}_{k=0}^{N-1}=\{(q_{k}^{0},...,q_{k}^{s})\}_{k=0}^{N-1}$ that satisfies  \eqref{eq:del_1}-\eqref{eq:del_2} and the transition condition is a \emph{solution of the discrete Euler-Lagrange equations}. We denote the left hand side of \eqref{eq:del_1} and \eqref{eq:del_2} by $D_\text{DEL}L_{d}(q_{k-1},q_{k})$ such that we have
\begin{eqnarray*}
\left(D_\text{DEL}L_{d}(q_{k-1},q_{k})\right)_{1} & = & D_{s+1}L_{d}(q_{k-1}^{0},...,q_{k-1}^{s})+D_{1}L_{d}(q_{k}^{0},...,q_{k}^{s})\\
\left(D_\text{DEL}L_{d}(q_{k-1},q_{k})\right)_{2} & = & D_{2}L_{d}(q_{k}^{0},...,q_{k}^{s})\\
 & \vdots\\
\left(D_\text{DEL}L_{d}(q_{k-1},q_{k})\right)_{s} & = & D_{s}L_{d}(q_{k}^{0},...,q_{k}^{s}).
\end{eqnarray*}
As in \cite{MaWe01} we can introduce the standard discrete Lagrangian that depends only on two configurations as
\[
L_{d}(q_{k}^{0},q_{k}^{s}) = L_d(q_k),
\]
where $q_k^1,\ldots,q_k^{s-1}$ are implicitly determined by satisfying the internal stage equations~\eqref{eq:del_2}.
Alternatively, one can characterize the discrete Lagrangian in the following way (see \cite{MaWe01}),
\[
L_{d}(q_{k}^{0},q_{k}^{s})=\begin{array}[t]{c}
\text{ext}\\
{\begin{subarray}{c}
q_{k}^{\nu}\in Q\\
\nu\in\{1,...,s-1\}
\end{subarray}}
\end{array}h\sum_{i=1}^{r}b_{i}L(q_{d}(c_{i}h;q_{k})),\dot{q}_{d}(c_{i}h;q_{k})),
\]
meaning that $(s-1)$ configurations $q_k^1,\ldots,q_k^{s-1}$ are determined by extremizing the discrete Lagrangian.
The Lagrangian $L_{d}(q_{k}^{0},q_{k}^{s})$ provides the same iteration scheme as the discrete Lagrangian $L_d(q_k)$.

Let $q_k=(q_k^0,\ldots,q_k^s)$ and let $\alpha(q_k,q_{k+1})=q_{k+1}$ be the translation operator and $\pi(q_k,q_{k+1})=q_k$ the projection operator. The \emph{discrete Lagrangian evolution operator} $X_{L_d}: Q^{s+1}\rightarrow Q^{s+1}\times Q^{s+1}$ satisfies
\[
X_{L_{d}}(q_{k-1})=(q_{k-1},q_{k})\,\,\,\mbox{with}\,\, q_{k-1}^{s}=q_{k}^{0},
\]
\[
\pi\circ X_{L_{d}}(q_{k-1})=q_{k-1} \quad\text{and}\quad D_{\text{DEL}}L_{d}\circ X_{L_{d}}(q_{k-1})=0.
\]
The \emph{discrete Lagrangian map} $F_{L_d}: Q^{s+1} \rightarrow Q^{s+1}$ is defined by
\[
F_{L_{d}}(q_{k-1})=\alpha\circ X_{L_{d}}(q_{k-1})=q_{k}
\]
and generates the sequence of configurations that is denoted as the solution of the Euler-Lagrange equations.
The \emph{discrete Legendre transforms} $\mathbb{F}^{\pm}L_d: Q\times Q\rightarrow T^*Q$ are defined as 
\[
\mathbb{F}^{-}L_{d}(q_{k}^{0},q_{k}^{s})=(q_{k}^{0},p_{k}^{0-})=(q_{k}^{0},-D_{1}L_{d}(q_{k}^{0},q_{k}^{s})),
\]
\[
\mathbb{F}^{+}L_{d}(q_{k-1}^{0},q_{k-1}^{s})=(q_{k}^{0},p_{k}^{0+})=(q_{k}^{0},D_{s+1}L_{d}(q_{k-1}^{0},q_{k-1}^{s})).
\]
From the discrete Euler-Lagrange equations it follows that along the solution of the discrete Euler-Lagrange equations we have
\[
p_k^0 := p_{k}^{0-}=p_{k}^{0+}.
\]
Note that the discrete Lagrangian flow is well-defined, if the discrete Lagrangian is regular, i.e., if the discrete Legendre transforms are local isomorphisms what is assumed in the following.
As shown in Fig.~\ref{fig:discretemaps}, the \emph{discrete Hamiltonian map} $\tilde{F}_{L_d}:T^*Q \rightarrow T^*Q$ is given by
\[
\tilde{F}_{L_{d}}=\mathbb{F}^{\pm}L_{d}\circ F_{L{}_{d}}\circ(\mathbb{F}^{\pm}L_{d})^{-1}.
\]
\begin{figure}[htbp]
\centering
\begin{tikzpicture}[fill=blue!20]
\path (2.5,4) node(a) {$(q_0^0,\ldots,q_0^s)$};
\path (7.5,4) node(b)  {$(q_1^0,\ldots,q_1^s)$};
\path (0,0) node(c) {$(q_0^0,p_0^0)$};
\path (5,0) node(d) {$(q_1^0,p_1^0)$};
\path (10,0) node(e) {$(q_2^0,p_2^0)$};
\path (5.0,4.0) node[anchor=south] (f) {$F_{L_d}$};
\path (1.25,2) node[anchor=east] (g) {$\mathbb{F}^{-}L_d$};
\path (3.9,2) node[anchor=west] (h) {$\mathbb{F}^{+}L_d$};
\path (6.3,2) node[anchor=west] (i) {$\mathbb{F}^{-}L_d$};
\path (8.9,2) node[anchor=west] (j) {$\mathbb{F}^{+}L_d$};
\path (2.5,0) node[anchor=north] (k) {$\tilde{F}_{L_d}$};
\path (7.5,0) node[anchor=north] (l) {$\tilde{F}_{L_d}$};
\draw[thick,black,|->] (a) -- (b);
\draw[thick,black,|->] (a)  -- (c);
\draw[thick,black,|->] (a)  -- (d);
\draw[thick,black,|->] (b)  -- (d);
\draw[thick,black,|->] (b)  -- (e);
\draw[thick,black,|->] (c)  -- (d);
\draw[thick,black,|->] (d)  -- (e);
\end{tikzpicture}
\caption{Correspondence between the discrete Lagrangian and the discrete Hamiltonian
map.}
\label{fig:discretemaps}
\end{figure}
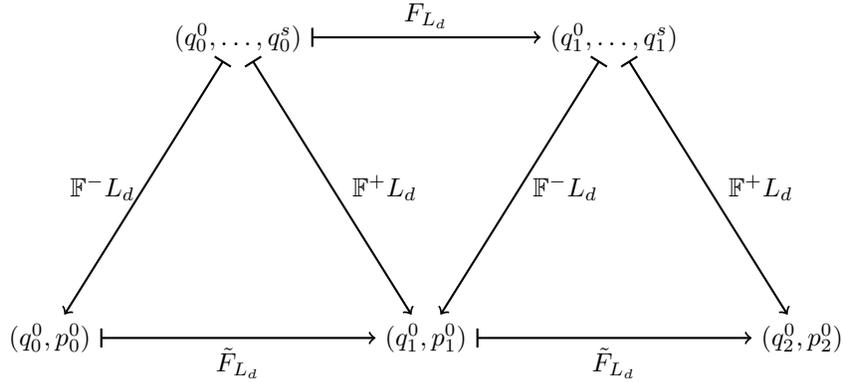

With the proposed method different variational integrators can be constructed. We use the following notation:
$(PsNrQu)$ is an integrator constructed as described above with a polyomial of degree $s$ with $s+1$ control points $(d_i)_{i=0}^s$ and a quadrature formula of order $u$ with $r$ quadrature points. Note that $u$ depends on $r$. If explicitly given, we denote by three letters the quadrature rule in use, i.e.~$Lob$ for Lobatto quadrature ($u=2r-2$) and $Gau$ for Gauss quadrature ($u=2r$).

\begin{example}
\begin{itemize}
\item[(i)] The integrator $(P1N1Q2Gau)$ is based on a polynomial with control points $d_0=0$, $d_1=1$ and the quadrature approximation
\begin{eqnarray*}
\int_{0}^{h}L(q(t),\dot{q}(t))\,dt\approx L_{d}((q_{0}^{0},q_{1}^{0}),h) & = & hL(q_d\left(h/2\right),\dot{q}_d\left(h/2\right))\\
 & = & hL\left(\frac{q_{0}^{0}+q_{1}^{0}}{2},\frac{q_{1}^{0}-q_{0}^{0}}{h}\right)
\end{eqnarray*}
and thus results in the midpoint rule discrete Lagrangian.
\item[(ii)] If the trapezoidal rule is applied as quadrature formula, we obtain the variational integrator $(P1N2Q2Lob)$ with discrete Lagrangian
\begin{eqnarray*}
L_{d}((q_{0}^{0},q_{1}^{0}),h) & = & \frac{h}{2}L(q_d(0),\dot{q}_d(0))+\frac{h}{2}L(q_d(h),\dot{q}_d(h))\\
 & = & \frac{h}{2}L\left(q_{0}^{0},\frac{q_{1}^{0}-q_{0}^{0}}{h}\right)+\frac{h}{2}L\left(q_{1}^{0},\frac{q_{1}^{0}-q_{0}^{0}}{h}\right)
\end{eqnarray*}
which is the discrete Lagrangian of the St\"ormer-Verlet method with factor $\frac{1}{2}$ for a Lagrangian of the form $L(q,\dot{q})= \frac{1}{2} \dot{q}^TM \dot{q} - V(q)$.
\item[(iii)] For a second order polynomial ($s=2$) with control points $d_0=0$, $d_1=\frac{1}{2}$ and $d_2=1$ and by applying Simpson's rule, which is a Lobatto quadrature of order four, we obtain the variational integrator $(P2N3Q4Lob)$ with discrete Lagrangian given by
\begin{eqnarray*}
L_{d}((q_{0}^{0},q_{0}^{1},q_{0}^{2}),h) & = & \frac{h}{6}L(q_d(0),\dot{q}_d(0))+\frac{2h}{3}L(q_d\left(h/2\right),\dot{q}_d\left(h/2\right))+\frac{h}{6}L(q_d(h),\dot{q}_d(h))\\
 & = & \frac{h}{6}L\left(q_{0}^{0},\frac{-3q_{0}^{0}+4q_{0}^{1}-q_{0}^{2}}{h}\right)
  +\frac{2h}{3}L\left(q_{0}^{1},\frac{q_{0}^{2}-q_{0}^{0}}{h}\right)
  +\frac{h}{6}L\left(q_{0}^{2},\frac{q_{0}^{0}-4q_{0}^{1}+3q_{0}^{2}}{h}\right).
\end{eqnarray*} 
\end{itemize}
\end{example}

\begin{remark}[Implementation]\label{rem:implem}
For given configurations $q_0^0,\ldots,q_{0}^{s-1},q_1^0$ we compute the unknown configuration $q_2^0$ by performing one step of the discrete Lagrangian evolution
\[
(q_{0}^{0},...,q_{0}^{s-1},q_{1}^{0})\overset{{F_{L_{d}}} }{\longrightarrow}  (q_{1}^{0},...,q_{1}^{s-1},q_{2}^{0}).
\]
To this end, the system of $s$ nonlinear equations~\eqref{eq:del_1}-\eqref{eq:del_2} is solved for the $s$ unknowns $(q_{1}^{1},...,q_{1}^{s-1},q_{2}^{0})$. For the numerical solution, a Newton method can be used. 
For given initial configuration $q(0)$ and momentum $p(0)$, in the first step, \eqref{eq:del_1} is replaced by the discrete Legendre transform 
\[
p(0)=-D_{1}L_{d}(q_{0}^{0},\ldots,q_{0}^{s})
\]
and the system of equations is solved for $(q_0^1,\ldots,q_1^0)$.
For the quadrature rules considered here, \eqref{eq:del_1}-\eqref{eq:del_2}  typically provide an implicit scheme for nonlinear systems that has to be solved by an iterative solver all at once. 
\end{remark}

\begin{remark}[Galerkin methods]\label{rem:Gal}
For the Galerkin variational integrators as introduced in \cite{MaWe01}, Section~2.2.6, the number of quadrature points of the quadrature formula $(b_i,c_i)_{i=1}^r$ is fixed to $r=s$, where $s$ is the degree of the polynomial $q_d$. In our notation that means that only the methods $(PsNsQu)$ are investigated, which are shown to be equivalent to partitioned symplectic Runge-Kutta methods.
In particular, it is pointed out that the integrator $(PsNsQ2sGau)$, which uses the Gauss quadrature formula, corresponds to the \emph{collocation Gauss-Legendre rule}, whereas $(PsNsQ2s-2Lob)$ yields the standard Lobatto IIIA-IIIB partitioned Runge-Kutta method. For both methods the order is determined by the order of quadrature rule, i.e., $(PsNsQ2sGau)$ is of order $2s$ and $(PsNsQ2s-2Lob)$ is of order $2s-2$ (cf.~\cite{HaLuWa}). Although the Gauss quadrature formula leads to higher order schemes, in particular for stiff systems the choice of a quadrature rule involving $c_s=1$ leads to better numerical performance (cf.~\cite{HaLuWa, MaWe01}). If, in addition, one wishes to use a symmetric quadrature rule, i.e., $c_1=0$, the Legendre-Lobatto quadrature rule provides the highest possible order.
 \end{remark}

\section{Properties of Galerkin variational integrators}\label{sec:prop}

In this section properties of the Galerkin variational integrators, such as symplecticity, momentum-preservation, time reversibility and linear stability are studied.

\subsection{Preservation properties of variational integrators}\label{subsec:prop}

As already mentioned in Section~\ref{sec:varmech} variational integrators are structure-preserving, in particular they are symplectic and momentum-preserving. In this section we briefly repeat the notion of symplecticity and first integrals and their discrete counterparts.

\subsubsection{Symplecticity and energy behavior}
 
In the case of conservative systems (as considered here), the flow on $T^*Q$ of the Euler-Lagrange equations preserves the canonical symplectic form $\Omega = dq^i \wedge dp^i = \textbf{d} \theta$ of the Hamiltonian system, where $\theta = p^i dq^i$ is the canonical one-form. 
It is well known (cf.~e.g.~\cite{MaWe01}) that variational integrators are symplectic, that is the same property holds for the discrete flow of the discrete Euler-Lagrange equations. As a consequence, the canonical discrete symplectic form $\Omega = dq_0^i \wedge dp_0^i$ is exactly preserved for the discrete solution which is in particular true for the variational integrators presented in this work.

By using techniques from backward error analysis it is shown (cf.~e.g.~\cite{HaLuWa}) that symplectic integrators have excellent energy properties, meaning that for long-time integrations there is no artificial energy growth or decay due to numerical errors. 
This property is demonstrated numerically in Section~\ref{sec:numex} and illustrates the advantage of variational integrators over e.g.~nonsymplectic Runge-Kutta integrators in particular for long-time simulations.

\subsubsection{Preservation of first integrals}
The Noether theorem provides first integrals of the Euler-Lagrange equations which are also called \emph{momentum maps}. 
The discrete Noether theorem states that these invariants are also preserved for the discrete solution. The following two theorems are taken from \cite{HaLuWa}.

\begin{theorem}[Noether theorem]\label{th:Noether}
Let $L(q,\dot{q})$ be a regular Lagrangian. Suppose $G=\{ g_v: v\in \mathbb{R} \}$ is a one-parameter group of transformations ($g_v \circ g_w = g_{v+w}$) which leaves the Lagrangian invariant such that
 \[
 L(g_v(q),g'_v(q)\dot{q}) = L(q,\dot{q})\quad \forall v\in \mathbb{R}\; \forall (q,\dot{q})\in TQ.
 \]
 Let $a(q)= \frac{d}{dv} g_v(q) |_{v=0}$ be defined as the vector field with flow $g_v(q)$. Then
 \begin{equation}\label{eq:inv}
 I(q,p) = p^T a(q)
 \end{equation}
 is a first integral of the Euler-Lagrange equations.
\end{theorem}
The discrete analog of the Noether theorem is stated as follows.
\begin{theorem}[Discrete Noether theorem]\label{th:discreteNoether}
Suppose the one-parameter group of transformations $G=\{ g_v: v\in \mathbb{R} \}$ leaves the discrete Lagrangian $L_d(q_0^0,q_1^0)$ invariant, that means
\[
L_d(g_v(q_0^0),g_v(q_1^0)) = L_d(q_0^0,q_1^0) \quad \forall v\in\mathbb{R}\;\forall (q_0^0,q_1^0)\in Q\times Q.
\]
Then \eqref{eq:inv} is an invariant of the discrete Hamiltonian map $\tilde{F}_{L_d}$, i.e., 
\[
I\circ \tilde{F}_{L_d}(q_k^0,p_k^0)= I(q_k^0,p_k^0).
\]
\end{theorem}
Proofs of Theorems~\ref{th:discreteNoether} and \ref{th:discreteNoether} can be found in \cite{HaLuWa}.

Note that the invariant of the discrete Hamiltonian map only equals the first integral of the Euler-Lagrange equations if the discrete Lagrangian $L_d$ inherits the invariance of the Lagrangian $L$. In the following we show under which condition this invariance is inherited. 

\begin{definition}[Equivariance]
Let $\{ g_v: v\in\mathbb{R} \}$ be a one-parameter group of transformation. The interpolation polynomial $q_d$ of $(PsNsQu)$ is \emph{equivariant w.r.t.~$\{ g_v: v\in\mathbb{R} \}$} if
\begin{equation}\label{eq:equiq}
g_v(q_d(t; (q_0^0,\ldots,q_0^s))) = q_d(t; (g_v(q_0^0),\ldots,g_v(q_0^s))),
\end{equation}
\begin{equation}\label{eq:equiv}
g'_v(q_d(t; (q_0^0,\ldots,q_0^s))) \dot{q}(t;(q_0^0,\ldots,q_0^s)) = \dot{q}_d(t; (g_v(q_0^0),\ldots,g_v(q_0^s))).
\end{equation}
\end{definition}
Note that \eqref{eq:equiv} follows from \eqref{eq:equiq} by applying the chain rule.

\begin{theorem}[Invariance of discrete Lagrangian]\label{th:invLd}
Suppose that the interpolation polynomial $q_d$ of $(PsNrQu)$ is equivariant and that the regular Lagrangian $L$ is invariant w.r.t.~the one-parameter group $G=\{ g_v: v\in\mathbb{R} \}$. Then the discrete Lagrangian $L_d$ of $(PsNrQu)$ is invariant w.r.t.~$G$.
\end{theorem}
\begin{proof}
Let $q_{0}=(q_{0}^{0},\ldots,q_{0}^{s})$ with $q_{0}^{s}=q_{1}^{0}$ and $g_{v}\cdot q_{0}=(g_{v}(q_{0}^{0}),...,g_{v}(q_{0}^{s-1}),g_{v}(q_{1}^{0}))$. Let $(b_{i},c_{i})_{i=1}^{r}$ be the quadrature formula that corresponds to $(PsNrQu)$. With the invariance of $L$ and the equivariance of $q_d$ we have that
\begin{eqnarray*}
L_{d}(g_{v}(q_{0}^{0}),g_{v}(q_{1}^{0})) & = & \begin{array}[t]{c}
\text{ext}\\
{\begin{subarray}{c}
g_{v}(q_{0}^{\nu})\in Q\\
\nu\in\{1,...,s-1\}
\end{subarray}}
\end{array}h\sum_{i=1}^{r}b_{i}L(q_{d}(c_{i}h;g_{v}\cdot q_{0}),\dot{q}_{d}(c_{i}h;g_{v}\cdot q_{0}))\\
 & = & \begin{array}[t]{c}
\text{ext}\\
{\begin{subarray}{c}
q_{0}^{\nu}\in Q\\
\nu\in\{1,...,s-1\}
\end{subarray}}
\end{array}h\sum_{i=1}^{r}b_{i}L(g_{v}(q_{d}(c_{i}h;q_{0})),g'_{v}(q_{d}(c_{i}h;q_{0}))\dot{q}_{d}(c_{i}h;q_{0}))\\
 & = & \begin{array}[t]{c}
\text{ext}\\
{\begin{subarray}{c}
q_{0}^{\nu}\in Q\\
\nu\in\{1,...,s-1\}
\end{subarray}}
\end{array}h\sum_{i=1}^{r}b_{i}L(q_{d}(c_{i}h;q_{0})),\dot{q}_{d}(c_{i}h;q_{0}))\\
 & = & L_{d}(q_{0}^{0},q_{1}^{0}).
\end{eqnarray*}
\end{proof}

A general form of the statement of Theorem~\ref{th:invLd} for Galerkin Lie group variational integrators can also be found in \cite{Leok2011}.

\begin{remark}\label{rem:lingroup}
From Theorem~\ref{th:invLd} it follows that for linear group transformations the discrete Lagrangian $L_d$ of $(PsNrQu)$ inherits the invariance of the Lagrangian $L$ (cf.~\cite{MaWe01}).
\end{remark}

The properties described in Section~\ref{subsec:prop} are valid for all variational integrators. In the following we present special properties of the variational integrators constructed in this work.

\subsection{Time reversibility}\label{subsec:timerevers}

A further geometric property of Hamiltonian systems is the time reversibility. 
It seems likely to use numerical methods that produce a reversible numerical flow when they are applied to a reversible Hamiltonian system. Furthermore, the numerical solution has a long-time behavior similar to the exact solution (see \cite{HaLuWa}), and thus time reversibility is a desirable property also for the variational integrators presented in this work. In the following, we repeat the definitions of time reversibility and the adjoint discrete Lagrangian (following \cite{HaLuWa} and \cite{MaWe01}) und show, under which conditions the variational integrators $(PsNrQu)$ are time-reversible.  

\begin{definition}[Time reversibility and adjoint (\cite{HaLuWa})]
A numerical one-step method $\Phi_d$ is called \emph{symmetric} or \emph{time-reversible} if it satisfies
\[\Phi_d^h \circ \Phi_d^{-h} = id \quad\text{or equivalently}\quad \Phi_d^h=(\Phi_d^{-h})^{-1}.
\]
The \emph{adjoint} method denoted by $\Phi_h^*$ is defined by
\[
(\Phi_d^h)^* = (\Phi_d^{-h})^{-1}.
\]
A method is called \emph{self-adjoint} if we have
\[
\Phi_d^*=\Phi_d.
\]
\end{definition}
Thus, symmetric, time-reversible, and self-adjoint are equivalent notions.
In the following, we show that the integrators $(PsNrQ2rGau)$ and $(PsNrQ2r-2Lob)$ are self-adjoint and thus time-reversible methods.
\begin{definition}[Adjoint of the discrete Lagrangian (\cite{MaWe01})]
The \emph{adjoint discrete Lagrangian} $L_d^*$ of the discrete Lagrangian is given by
\[
L_d^*(q_0^0,q_1^0,h)=-L_d(q_1^0,q_0^0,-h).
\]
The discrete Lagrangian is \emph{self-adjoint} if
\[
L_d(q_0^0,q_1^0,h)=L_d^*(q_0^0,q_1^0,h).
\]
\end{definition}
The following well-known theorem (see e.g.~\cite{MaWe01}) connects the adjoint of the discrete Lagrangian with the adjoint of the discrete Hamiltonian flow.
\begin{theorem}\label{th:adjointflow}
If the discrete Lagrangian $L_d$ has a discrete Hamiltonian
map $\tilde{F}_{L_d}$, then the discrete Hamiltonian map of the adjoint discrete Lagrangian $L_d^*$  
equals the adjoint map, i.e.~$\tilde{F}_{L_d^*}= \tilde{F}_{L_d}^*$.
If the discrete
Lagrangian is self-adjoint, then the method is self-adjoint. Conversely, if
the method is self-adjoint, then the discrete Lagrangian is equivalent\footnote{Two discrete Lagrangian are equivalent if their discrete Hamiltonian map are the same.} to a
self-adjoint discrete Lagrangian.
\end{theorem}
A proof of this theorem can be found in \cite{MaWe01}, Theorem 2.4.1.
To show the time reversibility, we first show that the discrete Lagrangian is self-adjoint.
\begin{theorem}\label{th:lagadj}
Let $L_d$ be the discrete Lagrangian of $(PsNrQu)$ with symmetric quadrature formula $(b_i,c_i)_{i=1}^r$ and interpolation polynomial $q_d$ with symmetrically distributed control points $(d_i)_{i=0}^s$, i.e., $b_i = b_{r+1-i}$, $c_i+c_{r+1-i}=1$, $i=1,\ldots,r$, and $d_i = 1-d_{s-i}$, $i=0,\ldots,s$. Then $L_d$ is self-adjoint.
\end{theorem}
\begin{proof}
We have
\begin{eqnarray}
q_{d}(-t;q_{0},-h) & = & \sum_{\nu=0}^{s}q_{0}^{\nu}l_{\nu,s}\left(t/h\right)=q_{d}(t;q_{0},h),\label{eq:selbstad}\\
\dot{q}_{d}(-t;q_{0},-h) & = & \frac{1}{-h}\sum_{\nu=0}^{s}q_{0}^{\nu}\dot{l}_{\nu,s}\left(t/h\right)=-\dot{q}_{d}(t;q_{0},h).\label{eq:selbstadII}
\end{eqnarray}
Due to the symmetry of the control points $d_i$ it follows for the Lagrange-polynomials $l_{\nu,s}$ with $\tau \in [0,1]$
\begin{eqnarray*}
l_{\nu,s}(1-\tau) & = & \prod_{0\leq i\leq s,i\neq \nu}\frac{(1-\tau)-d_{i}}{d_{\nu}-d_{i}}=  \prod_{0\leq i\leq s,i\neq \nu}\frac{(-\tau)+(1-d_{i})}{(-1+d_{\nu})+(1-d_{i})}\\
 & = & \prod_{0\leq i\leq s,i\neq \nu}\frac{(-\tau)+d_{s-i}}{-d_{s-\nu}+d_{s-i}} =  \prod_{0\leq i\leq s,i\neq \nu}\frac{\tau-d_{s-i}}{d_{s-\nu}-d_{s-i}}\\
 & = & l_{s-\nu,s}(\tau).
\end{eqnarray*}
With $q_{0}=(q_{0}^{0},...,q_{0}^{s})$, $q_{0}^{s}=q_{1}^{0}$, and $\tilde{q}_{0}=(q_{0}^{s},...,q_{0}^{0})$ we have for $t\in[0,h]$
\begin{eqnarray*}
q_{d}(t,\tilde{q}_{0},h) & = & \sum_{\nu=0}^{s}q_{0}^{s-\nu}l_{\nu,s}\left(\frac{t}{h}\right) =  \sum_{k=0}^{s}q_{0}^{k}l_{s-k,s}\left(\frac{t}{h}\right)\\
 & = & \sum_{k=0}^{s}q_{0}^{k}l_{k,s}\left(1-\frac{t}{h}\right) = \sum_{k=0}^{s}q_{0}^{k}l_{k,s}\left(\frac{h-t}{h}\right)\\
 & = & q_{d}(h-t;q_{0},h)
\end{eqnarray*}
and by taking the time derivative we have
\begin{eqnarray*}
\dot{q}_{d}(h-t;q_{0},h) & = & -\dot{q}_{d}(t;\tilde{q}_{0},h).
\end{eqnarray*}
Thus, together with \eqref{eq:selbstad}-\eqref{eq:selbstadII} we obtain
\begin{eqnarray*}
q_{d}(-t;\tilde{q}_{0},-h) & = & q_{d}(t;\tilde{q}_{0},h) =  q_{d}(h-t;q_{0},h),\\
\dot{q}_{d}(-t;\tilde{q}_{0},-h) & = & -\dot{q}_{d}(t;\tilde{q}_{0},h)= \dot{q}_{d}(h-t;q_{0},h).
\end{eqnarray*}
By substituting these expressions in the adjoint discrete Lagrangian we have with the symmetry of the quadrature formula $(b_i,c_i)_{i=0}^r$
\begin{eqnarray*}
-L_{d}(q_{1}^{0},q_{0}^{0},-h) & = & \begin{array}[t]{c}
\text{ext}\\
{\begin{subarray}{c}
q_{0}^{\nu}\in Q\\
\nu\in\{1,...,s-1\}
\end{subarray}}
\end{array}-(-h)\sum_{i=1}^{r}b_{i}L(q_{d}(-c_{i}h;\tilde{q}_{0},-h)),\dot{q}_{d}(-c_{i}h;\tilde{q}_{0},-h))\\
 & = & \begin{array}[t]{c}
\text{ext}\\
{\begin{subarray}{c}
q_{0}^{\nu}\in Q\\
\nu\in\{1,...,s-1\}
\end{subarray}}
\end{array}h\sum_{i=1}^{r}b_{i}L(q_{d}(h-c_{i}h;q_{0},h)),\dot{q}_{d}(h-c_{i}h;q_{0},h))\\
 & = & \begin{array}[t]{c}
\text{ext}\\
{\begin{subarray}{c}
q_{0}^{\nu}\in Q\\
\nu\in\{1,...,s-1\}
\end{subarray}}
\end{array}h\sum_{i=1}^{r}b_{r+1-i}L(q_{d}(c_{r+1-i}h;q_{0},h)),\dot{q}_{d}(c_{r+1-i}h;q_{0},h))\\
 & = & L_{d}(q_{0}^{0},q_{1}^{0}).
\end{eqnarray*}

\end{proof}

\begin{theorem}
The discrete Hamiltonian maps of $(PsNrQ2rGau)$ and $(PsNrQ2r-2Lob)$ are self-adjoint and thus time-reversible.
\end{theorem}
\begin{proof}
Since the Gauss and the Lobatto quadrature are symmetric and since the control points of $q_d$ are chosen symmetrically, it follows with Theorem~\ref{th:lagadj} that the discrete Lagrangian of $(PsNrQ2rGau)$ and $(PsNrQ2r-2Lob)$ are self-adjoint. The statement follows with Theorem~\ref{th:adjointflow}. \end{proof}

\subsection{Linear stability analysis}\label{subsec:linstab}

In the following, we investigate the stability properties of the constructed variational integrators. We restrict ourselves to a linear stability analysis. Following \cite{Leimkuhler2004}, we consider the Lagrangian of a harmonic oscillator 
\begin{equation}\label{eq:Lag_ho}
L(q,\dot{q}) = \frac{1}{2} \dot{q}^2 - \frac{1}{2} \omega^2 q^2
\end{equation}
with $\omega, q \in \mathbb{R}$.
The Hamiltonian equations read
\begin{align*}
\dot{p}& = -\omega^2 q,\quad \dot{q}= p
\end{align*}
and the exact solution is given by
\begin{eqnarray*}
\left(\begin{array}{c}
p(t)\\
\omega q(t)
\end{array}\right) & = & \left(\begin{array}{cc}
\cos{(\omega t)} & -\sin{(\omega t)}\\
\sin{(\omega t)} & \cos{(\omega t)}
\end{array}\right)\left(\begin{array}{c}
p(0)\\
\omega q(0)
\end{array}\right)= M_\omega \left(\begin{array}{c}
p(0)\\
\omega q(0)
\end{array}\right) 
\end{eqnarray*}
with $\text{det}(M_\omega)=1$. The eigenvalues of $M_\omega$ are $\lambda_{1,2}=e^{\pm iwt}$ and thus, we have that $|\lambda_{1,2}|=1.$

By applying a variational integrator $(PsNrQu)$ to the Lagrangian~\eqref{eq:Lag_ho}, the discrete Euler-Lagrange equations form a linear system of equations such that the discrete Hamiltonian map $\tilde{F}_{L_d}$ of $(PsNrQu)$ can be written as linear map
\begin{eqnarray*}
\left(\begin{array}{c}
p_{1}^{0}\\
wq_{1}^{0}
\end{array}\right) & = & M_{h,w}\left(\begin{array}{c}
p_{0}^{0}\\
\omega q_{0}^{0}
\end{array}\right)
\end{eqnarray*}
with matrix $M_{h,\omega}$.

Following \cite{Leimkuhler2004}, we call a numerical solution \emph{asymptotically stable} if the growth of the solution is asymptotically bounded. A sufficient condition for asymptotic stability is that the eigenvalues of $M_{h,\omega}$ are in the unit desk of the complex plane  and are simple if they lie on the unit circle.
We investigate this property for selected variational integrators.

\begin{enumerate}
\item For the midpoint rule $(P1N1Q2Gau)$ we have
\[
M_{h,\omega}=\left(\begin{array}{cc}
-\frac{h^{2} \omega^{2}-4}{h^{2} \omega^{2}+4} & -\frac{4 h \omega}{h^{2} \omega^{2}+4}\\
\frac{4 h \omega}{h^{2} \omega^{2}+4} & -\frac{h^{2} \omega^{2}-4}{h^{2} \omega^{2}+4}
\end{array}\right).
\]
Since $M_{h,\omega}M_{h,\omega}^T = \text{Id}$, $M_{h,\omega}$ is orthogonal with  $|\lambda(M_{h,\omega})| = 1$. Thus, the midpoint rule is asymptotically stable for all $h,\omega \in \mathbb{R}$.
\item The St\"ormer-Verlet method $(P1N2Q2Lob)$ has the iteration matrix
\[
M_{h,\omega}=\left(\begin{array}{cc}
1-\frac{h^{2}\omega^{2}}{2} & \frac{h^{3} \omega^{3}}{4}-h \omega\\
h \omega & 1-\frac{h^{2} \omega^{2}}{2}
\end{array}\right).
\]
with eigenvalues
\[
\lambda_{1,2}=1-\frac{h^{2} \omega^{2}}{2}\pm\frac{h \omega\sqrt{h^{2} \omega^{2}-4}}{2},
\]
thus, the method is stable for $(h\omega)^{2}<4$ (cf.~\cite{Leimkuhler2004}).
\item For the $(P2N2Q4Gau)$ method we have
\[
M_{h,\omega}=\left(\begin{array}{cc}
\frac{h^{4}  \omega^{4}-60  h^{2}  \omega^{2}+144}{h^{4}  \omega^{4}+12  h^{2}  \omega^{2}+144} & -\frac{12  h  \omega \left(h^{2}  \omega^{2}-12\right)}{h^{4}  \omega^{4}+12  h^{2}  \omega^{2}+144}\\
\frac{12  h  \omega \left(h^{2}  \omega^{2}-12\right)}{h^{4}  \omega^{4}+12  h^{2}  \omega^{2}+144} & \frac{h^{4}  \omega^{4}-60  h^{2}  \omega^{2}+144}{h^{4}  \omega^{4}+12  h^{2}  \omega^{2}+144}
\end{array}\right),
\]
and the method is stable for all $q,\omega\in\mathbb{R}$ since $M_{h,\omega}$ is orthogonal.
\item The $(P2N3Q4Lob)$ scheme results in the iteration matrix
\[
M_{h,\omega}=\left(\begin{array}{cc}
\frac{h^{4}  \omega^{4}-22  h^{2}  \omega^{2}+48}{2  h^{2}  \omega^{2}+48} & \frac{24  h  \omega-3  h^{3}  \omega^{3}}{h^{2}  \omega^{2}+24}\\
-\frac{h  \omega \left(h^{4}  \omega^{4}-36  h^{2}  \omega^{2}+288\right)}{12  h^{2}  \omega^{2}+288} & \frac{h^{4}  \omega^{4}-22  h^{2}  \omega^{2}+48}{2  h^{2}  \omega^{2}+48}
\end{array}\right)
\]
with eigenvalues
\[
\lambda_{1,2}=\frac{h^{4} \omega^{4}-22 h^{2} \omega^{2}+48\pm h \omega\sqrt{h^{6} \omega^{6}-44\, h^{4} \omega^{4}+576 h^{2} \omega^{2}-2304}}{2 h^{2}\, \omega^{2}+48}.
\]
The stability region is shown in Fig.~\ref{fig:stab_P2N3Lob}. The integrator is stable for $| h\omega | < 2 \sqrt{2}$.
\begin{figure}
\centering 
\def\svgwidth{8cm} 
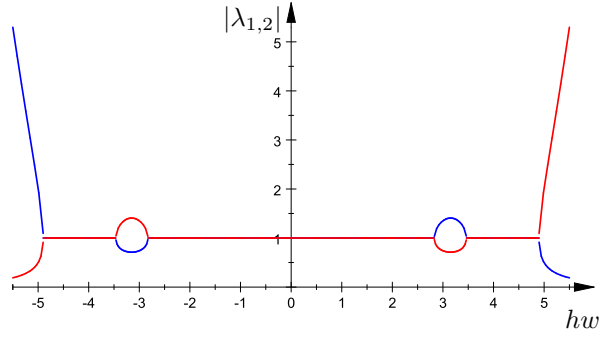
 \caption{Modulus of the eigenvalues of $M_{h,\omega}$ for the $(P2N3Q4Lob)$ integrator. The method is stable for $| h\omega | < 2 \sqrt{2}$. }
 \label{fig:stab_P2N3Lob}
\end{figure}
\item The iteration matrix for $(P3N3Q6Gau)$ reads
\[
M_{h,\omega}=\left(\begin{array}{cc}
-\frac{h^{6}  \omega^{6}-264  h^{4}  \omega^{4}+6480  h^{2}  \omega^{2}-14400}{h^{6}  \omega^{6}+24  h^{4}  \omega^{4}+720  h^{2}  \omega^{2}+14400} & \frac{24  h  \omega \left(h^{4}  \omega^{4}-70  h^{2}  \omega^{2}+600\right)}{h^{6}  \omega^{6}+24  h^{4}  \omega^{4}+720  h^{2}  \omega^{2}+14400}\\
-\frac{24  h  \omega \left(h^{4}  \omega^{4}-70  h^{2}  \omega^{2}+600\right)}{h^{6}  \omega^{6}+24  h^{4}  \omega^{4}+720  h^{2}  \omega^{2}+14400} & -\frac{h^{6}  \omega^{6}-264  h^{4}  \omega^{4}+6480  h^{2}  \omega^{2}-14400}{h^{6}  \omega^{6}+24  h^{4}  \omega^{4}+720  h^{2}  \omega^{2}+14400}
\end{array}\right).
\]
The scheme is again stable for all $q,\omega\in\mathbb{R}$ due to the orthogonality of $M_{h,\omega}$.
\item The integrator $(P3N4Q6Lob)$ gives
\[
M_{h,\omega}=\left(\begin{array}{cc}
-\frac{\frac{h^{6}  \omega^{6}}{2}-46  h^{4}  \omega^{4}+840  h^{2}  \omega^{2}-1800}{h^{4}  \omega^{4}+60  h^{2}  \omega^{2}+1800} & \frac{6  h  \omega \left(h^{4}  \omega^{4}-40  h^{2}  \omega^{2}+300\right)}{h^{4}  \omega^{4}+60  h^{2}  \omega^{2}+1800}\\
\frac{h  \omega \left(h^{6}  \omega^{6}-144  h^{4}  \omega^{4}+5760  h^{2}  \omega^{2}-43200\right)}{24 \left(h^{4}  \omega^{4}+60  h^{2}  \omega^{2}+1800\right)} & -\frac{h^{6}  \omega^{6}-92  h^{4}  \omega^{4}+1680  h^{2}  \omega^{2}-3600}{2  h^{4}  \omega^{4}+120  h^{2}  w^{2}+3600}
\end{array}\right).
\]
The eigenvalues are given as
$$\lambda_{1,2}=\frac{3600-1680x^{2}+92x^{4}-x^{6}\pm x\sqrt{x^{10}-184x^{8}+11820x^{6}-316800x^{4}+3456000x^{2}-12960000}}{2x^{4}+120x^{2}+3600}$$
with $x:=h\omega$.
In Figs~\ref{fig:stab_P3N4Lob} and \ref{fig:stab_P3N4Lob_zoom} (zoom of Fig.~\ref{fig:stab_P3N4Lob}) the stability region is shown.
\begin{figure}
\centering 
\def\svgwidth{8cm} 
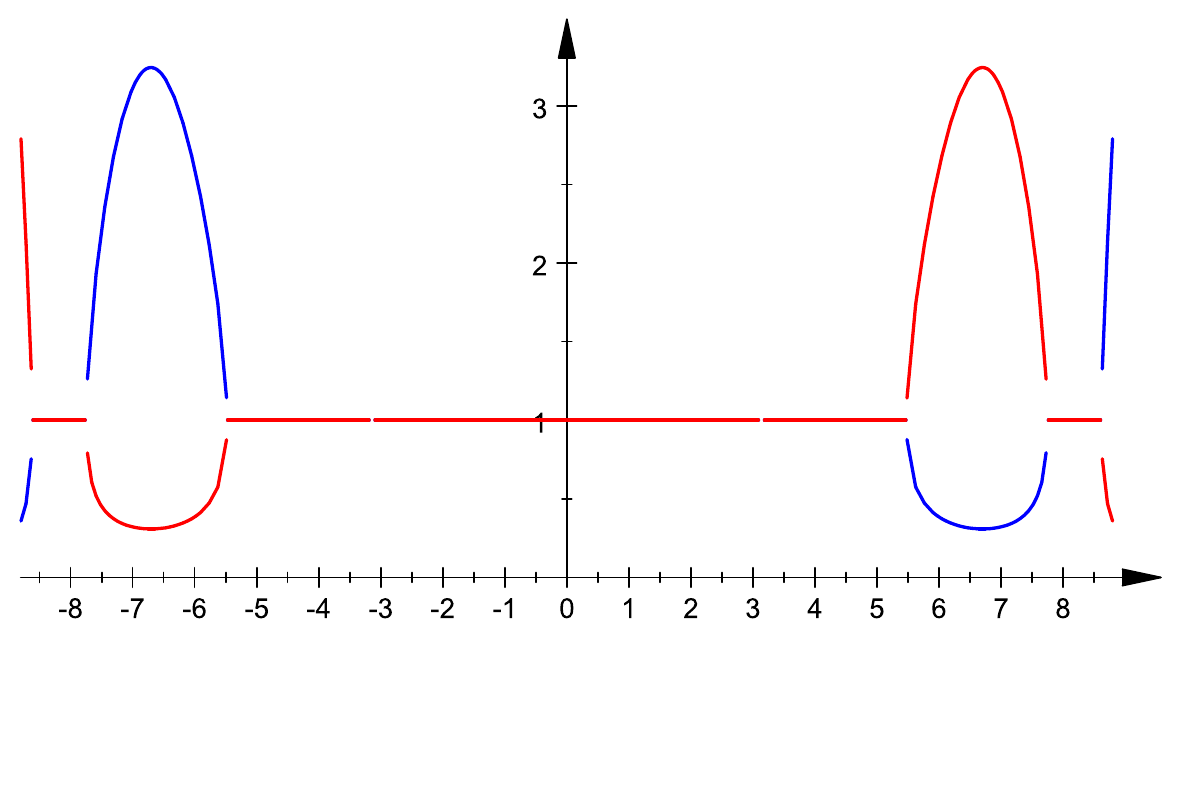
\caption{Modulus of the eigenvalues of $M_{h,\omega}$ for the $(P3N4Q6Lob)$ integrator.}
 \label{fig:stab_P3N4Lob}
\end{figure}
\begin{figure}
\centering 
\def\svgwidth{8cm} 
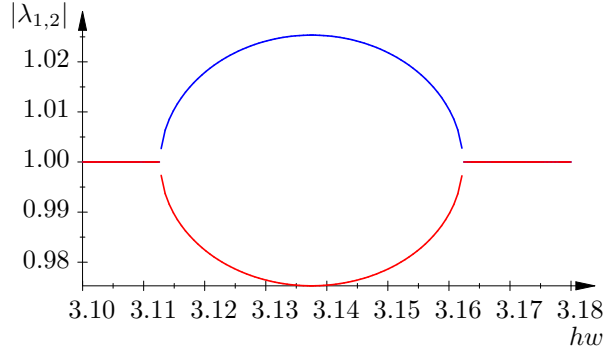
\caption{Modulus of the eigenvalues of $M_{h,\omega}$ for the $(P3N4Q6Lob)$ integrator (zoom).}
\label{fig:stab_P3N4Lob_zoom}
\end{figure}
\end{enumerate}
An interesting observation is that for the integrators $(PsNsQ2sGau)$, $s=\{1,2,3\}$ the iteration matrices $M_{h,\omega}$ are orthogonal independent of the step size $h$ and thus asymptotically stable.
Indeed, we can show that this unrestricted stability property holds for general integrators of this type.
\begin{lemma}[A-stability of $(PsNsQ2sGau)$]
The variational integrator $(PsNsQ2sGau)$ is A-stable.
\end{lemma}
\begin{proof}
In \cite{MaWe01} it is shown that the integrator $(PsNsQ2sGau)$ is equivalent to the collocation Gauss-Legendre rule (see also Remark~\ref{rem:Gal}) which is A-stable (see \cite{HaWa}, Theorem 5.2). 
\end{proof}

\section{Numerical convergence analysis}\label{sec:numex}

In this section we numerically analyze the convergence order of the constructed variational integrators $(PsNrQu)$ for $s,r,u\in\mathbb{N}$ and the theoretical results on the preservation properties are evaluated. 
To this end, we consider two examples, the harmonic oscillator and the Kepler problem and we want to numerically determine the maximal order of the variational integrator $(PsNrQu)$.

In Section~\ref{sub:discrteHam} we assume that the discrete Lagrangian is regular to obtain a well-defined discrete Lagrangian flow.
In \cite{HaLe13} it is shown that the well-posedness depends on the order of the quadrature rule that is used to approximate the action integral.
In particular, it is shown that for a Lagrangian of the form $L= \frac{1}{2} \dot{q}^T M \dot{q}-V(q)$ with $M$ symmetric positive-definite and $\nabla V$ Lipschitz continuous a unique solution of the internal stage equations \eqref{eq:del_2} exists if the used quadrature rule is of order at least $2s-1$ (with $s$ the degree of the interpolation polynomial), i.e.~$u\ge 2s-1$.
For the Gauss and Lobatto quadrature this means that we have to choose $r\ge s$ and $r\ge s+1$, respectively, since this yields quadrature rules of order $2s$.
Note that the variational integrator $(PsNsQ2s-2Lob)$ with $r=s$ which yields the standard Lobatto IIIA-IIIB partitioned Runge-Kutta method does not satisfy the condition $u\ge 2s-1$, however, it is only a sufficient not a necessary condition for uniqueness of solutions.
Thus, to ensure a well-defined discrete Lagrangian flow, we restrict to variational integrators with quadrature rule of Gauss or Lobatto type and for which the polynomial degree is smaller or equal to $r$ ($s\le r$).
%
%
The implementation is performed as described in Remark~\ref{rem:implem}.

\subsection{Harmonic oscillator}\label{subsec:harmosc}

Consider the two-dimensional harmonic oscillator with mass being equal to one and the Lagrangian $L(q,\dot{q}) = \frac{1}{2} \dot{q}^T \dot{q} - \frac{1}{2} q^T q$ with $q,\dot{q}\in\mathbb{R}^2$. The Euler-Lagrange equations are
\[
\ddot{q}(t)=-q(t).
\]
By the Legendre transform we obtain the Hamiltonian system
\begin{equation}\label{Ham:harmos}
\dot{q}(t) = p(t), \quad \dot{p}(t) = -q(t).
\end{equation}
The total energy of the system is given by the Hamiltonian $H(q,p) = \frac{1}{2} p^T p + \frac{1}{2} q^T q$ with $q,p\in\mathbb{R}^2$.

\subsubsection{Numerical convergence order}

Let $(q^e,p^e)(t)$ denote the exact solution consisting of configuration and momentum of the Hamiltonian system~\eqref{Ham:harmos}.
For the error calculations we use the global error determined by
\begin{equation}
\max_{\begin{subarray}{c}
k\in\{0,...,N\}\\
i\in\{1,2\}
\end{subarray}}|q_{k,i}^{0}-q_{i}^{e}(hk)|,\quad\quad
\max_{\begin{subarray}{c}
k\in\{0,...,N\}\\
i\in\{1,2\}
\end{subarray}}|p_{k,i}^{0}-p_{i}^{e}(hk)|
\end{equation}
with step size $h$ and the index $i$ denotes the components of the configuration and momenta, respectively.
$(q_k^0,p_k^0)_{k=0}^N$ is the discrete solution computed by a variational integrator with $q_N^0 = q_{N-1}^s$ and $p_N^0 = p_{N-1}^s$. 
In Figs~\ref{fig:Gaukonvplot} and \ref{fig:Lobkonvplot} the
global error for the different variational integrators $(PsNrQu)$ is shown
in dependence of the step size $h$.
The numerically determined order is given in Figs~\ref{fig:Ordnung-der-variatinellen}
and \ref{fig:Ordnung-der-variatinellen-1} and summarized in the
Tables~\ref{tab:Gau} and \ref{tab:Lob}\footnote{Note that in Table~\ref{tab:Lob} the numerical results for $r=2$ are not shown in Fig.~\ref{fig:Lobkonvplot} and \ref{fig:Ordnung-der-variatinellen-1}.}.
Note that the values for $r=s$ in both tables correspond to the convergence orders of the collocation Gauss-Legendre rule and the Lobatto IIIA-IIIB partitioned Runge-Kutta method, respectively. 

The first observation by considering the values in the Tables~\ref{tab:Gau} and \ref{tab:Lob} is that the convergence order for all combinations of polynomial degree $s$ and number $r$ of quadrature points can be determined as $\min{(2s,u)}$ with $u=2r$ and $u=2r-2$ being the order of the Gauss and Lobatto quadrature rule, respectively.
In particular, for the Galerkin variational integrators based on the Gauss quadrature formula (Table~\ref{tab:Gau}), for a given polynomial degree $s$ the order of convergence can not be improved for any quadrature rule based on more quadrature points than $s$. Thus, 
a reasonable combination of polynomial degree and number of quadrature points is $r\ge s$. 
However, for the Galerkin variational integrators based on the Lobatto quadrature formula (Table~\ref{tab:Lob}), the order of convergence can be improved if the number of quadrature points is increased by at least one, i.e.~the best combinations satisfy $r\ge s+1$.
This demonstrates that the order of $(PsNsQ2s-2Lob)$ increases to $2s$ if a Lobatto quadrature with $s+1$ quadrature points is used, or, the order does not decrease if a polynomial of degree $s-1$ instead of $s$ is used.

Based on this numerical results we can conclude that for a variational integrator based on an approximation space of degree $s$ polynomials the convergence order $2s$ is possible, i.e.~the integrator is superconvergent. This maximal order is reduced if the quadrature rules used for the approximation of the action is not accurate enough. Thus, to guarantee the maximal convergence order we have to choose $u\ge 2s$.

\begin{figure}[h]
\centering
\includegraphics[width=\textwidth]{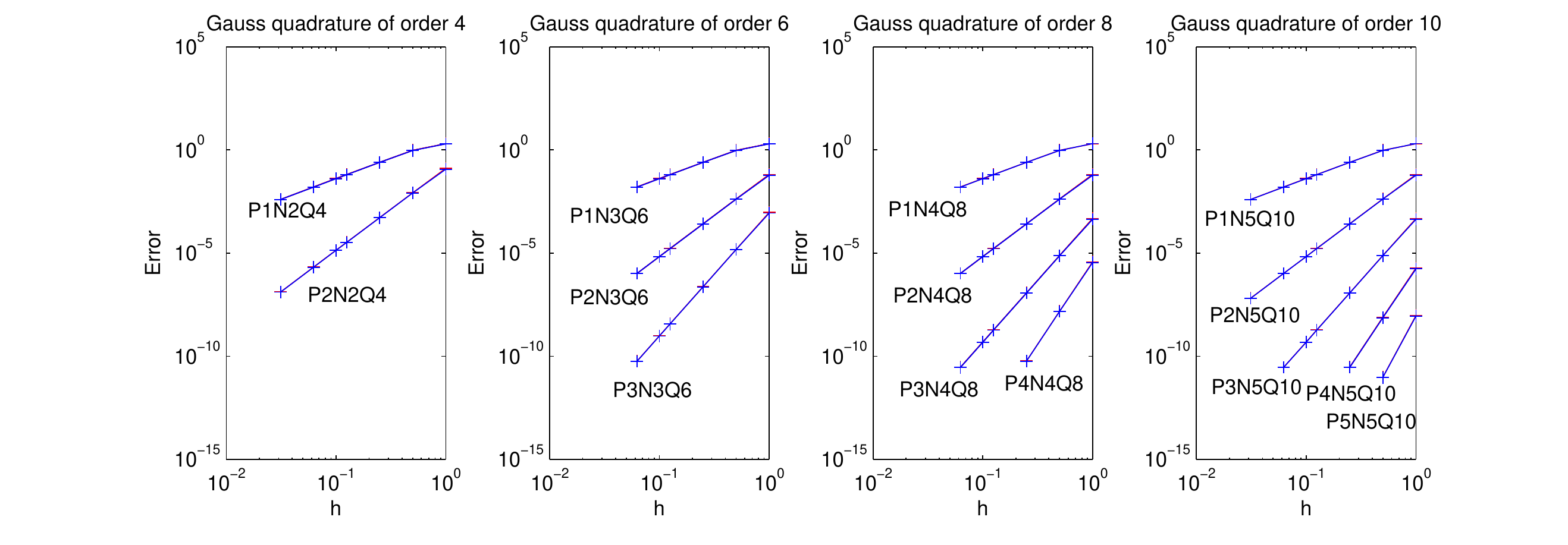}
\caption{\label{fig:Gaukonvplot}{Harmonic oscillator: Log-log plot
 of the error for the configurations $q$ and the momenta $p$
(superimposed) by step size $h$; with the use of variational integrators
$(PsNrQ2rGau)$; divided into four subplots, separated by applied
Gauss quadrature.}}
\end{figure}
\begin{figure}[h]
\centering
\includegraphics[width=\textwidth]{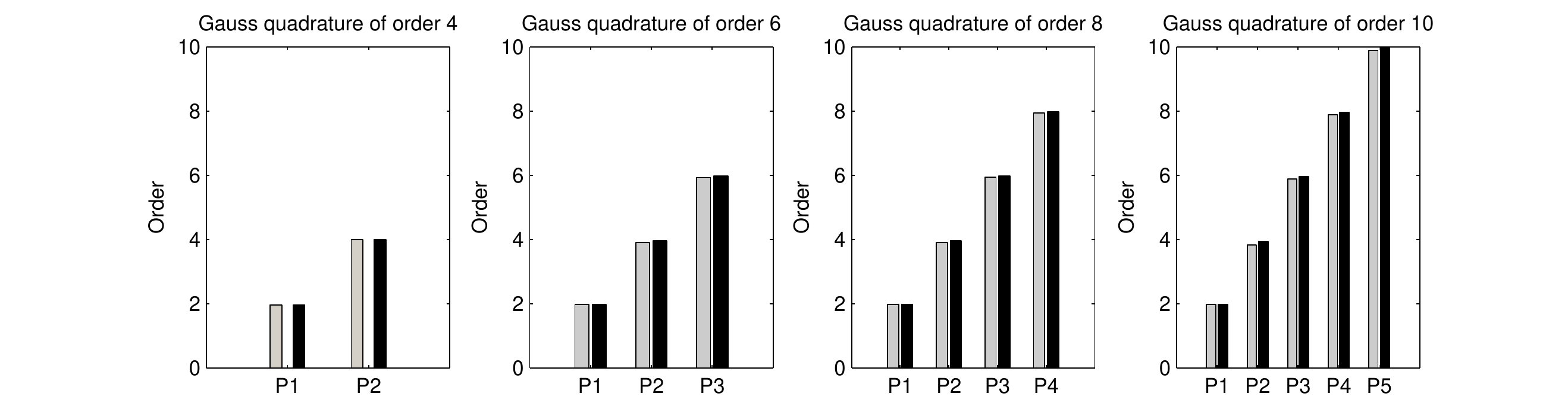}
\caption{\label{fig:Ordnung-der-variatinellen}{Harmonic oscillator:
Order of the variational integrators $\left(PsNrQ2rGau\right)$ with
respect to the configurations $q$ (left bar) and the momenta $p$
(right bar). The title gives the order of the applied Gauss quadrature
which is $2r$. The $x$-axis indicates the degree of the used polynomial.}}
\end{figure}
\begin{figure}[h]
\centering
\includegraphics[width=\textwidth]{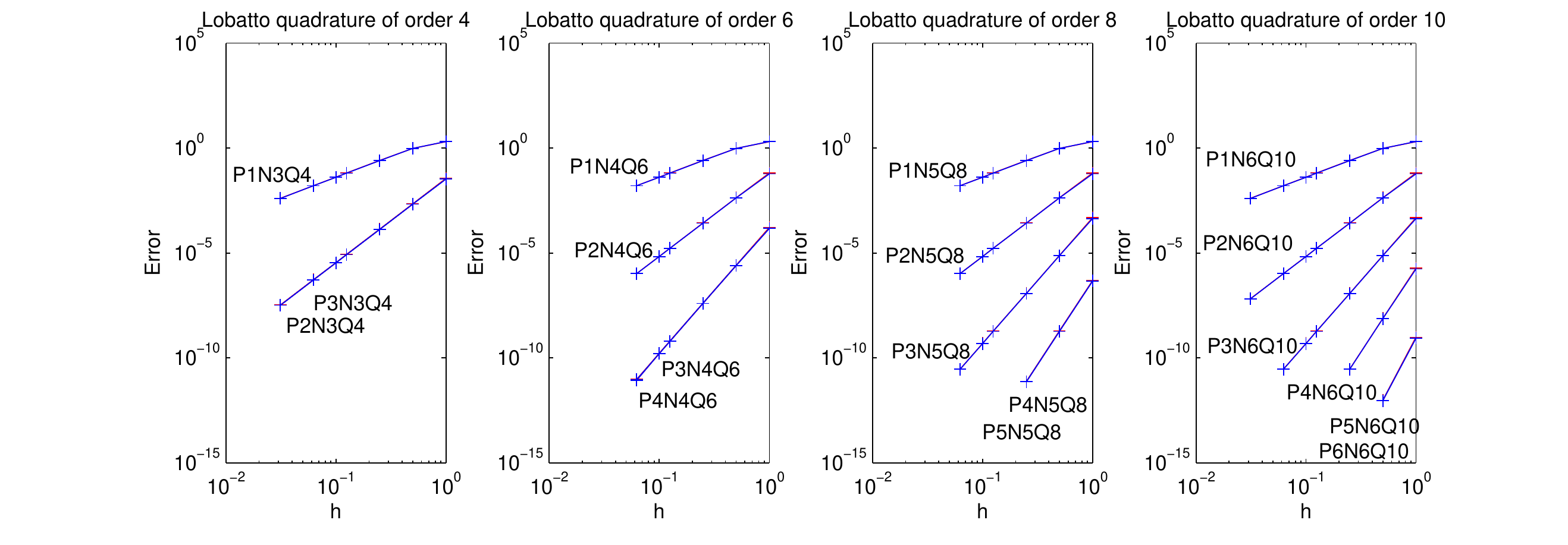}
\caption{\label{fig:Lobkonvplot}{Harmonic oscillator: Log-log plot of the error for the configurations $q$ and the momenta $p$
(superimposed) by step size $h$; with the use of variational integrators
$(PsNrQ2r-2Lob)$; divided into four subplots, separated by applied
Lobatto quadrature. Note that $(Ps-1NsQ2s-2Lob)$ and $(PsNsQ2s-2Lob)$ are of
the same order.}}
\end{figure}
\begin{figure}[h]
\centering
\includegraphics[width=\textwidth]{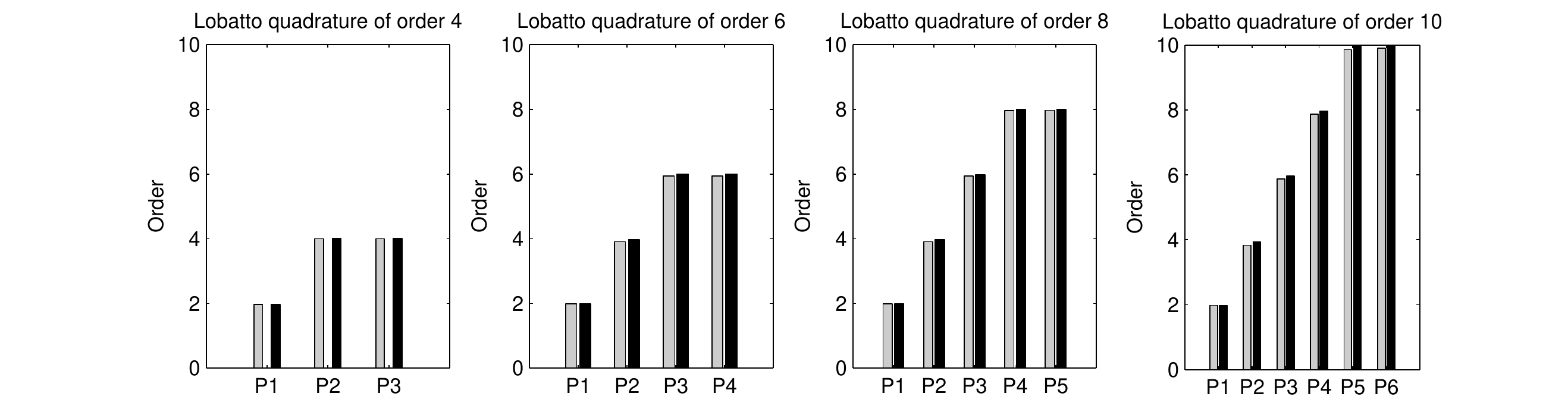}
\caption{\label{fig:Ordnung-der-variatinellen-1}{Harmonic oscillator:
Order of the variational integrators $\left(PsNrQ2r-2Lob\right)$ with
respect to the configurations $q$ (left bar) and the momenta $p$
(right bar). The title gives the order of the applied Lobatto quadrature
which is $2r-2$. The $x$-axis indicates the degree of the used polynomial.}}
\end{figure}

\begin{table}
\centering
\caption{Numerical convergence order of $\left(PsNrQ2rGau\right)$}
{\begin{tabular}{c|ccccc}
\hline
$r\backslash s$ & $s=1$ & $s=2$ & $s=3$ & $s=4$ & $s=5$ \tabularnewline
\hline 
$r=2$ & $2$& $4$ &  &  &   \tabularnewline
$r=3$ & $2$&$4$ & $6$ &  &   \tabularnewline
$r=4$ & $2$&$4$ & $6$ & $8$ &   \tabularnewline
$r=5$ & $2$&$4$ & $6$ & $8$ & $10$  \tabularnewline
\end{tabular}}
\label{tab:Gau}
\end{table}

\begin{table}
\centering
\caption{Numerical convergence order of $(PsNrQ2r-2Lob)$}
{\begin{tabular}{c|cccccc}
\hline
$r\backslash s$ & $s=1$ & $s=2$ & $s=3$ & $s=4$ & $s=5$ & $s=6$\tabularnewline
\hline 
$r=2$ &$2$& $2$ &  &  &  & \tabularnewline
$r=3$ & $2$&$4$ & $4$ &  &  & \tabularnewline
$r=4$ & $2$&$4$ & $6$ & $6$ &  & \tabularnewline
$r=5$ & $2$&$4$ & $6$ & $8$ & $8$ & \tabularnewline
$r=6$ & $2$ &$4$ & $6$ & $8$ & $10$ & $10$\tabularnewline
\end{tabular}}
\label{tab:Lob}
\end{table}


\subsubsection{Computational efficiency}
The numerical results also indicate which combination of quadrature rule and polynomial degree leads to lowest computational effort for a given order.
In Fig.~\ref{fig:Laufzeit_Fehler} the run-time compared with the
global error is given for the different integrators $(PsNrQ2r-2Lob)$ with
$s=\{1,...,6\}$ and $r=\{3,...,6\}$ and for different step sizes $h\in\{1,0.5,0.25,0.125,0.1,0.0625,0.03125\}$. The step size is
reduced until the error is below $10^{-10}$. 
The integrators $(P2N3Q4Lob)$,
$(P2N4Q6Lob)$, $(P2N5Q8Lob)$, and $(P2N6Q10Lob)$, which are of order four,
demonstrate that for an increasing number of nodes also the run-time increases.
The same behavior is also observable for the other integrators.

Further, a higher polynomial degree leads to an increasing number of Euler-Lagrange equations that have to be solved for and thus to higher computational effort.
In Fig.~\ref{fig:Laufzeit_Fehler}
it is shown that this leads to an increasing run-time (compare for example
$(P2N3Q4Lob)$ and $(P3N3Q4Lob)$ which are both of order four; or $(P5N6Q10Lob)$
and $(P6N6Q10Lob)$, which are of order ten).
Since the order of the integrator is $\min{(2s,u)}$, a reasonable choice for the polynomial degree is half of the order of the quadrature formula, i.e., $u=2s$. This guarantees a minimal number of discrete Euler-Lagrange equations without an order reduction of the variational integrator. For the Gauss and Lobatto quadrature rule the optimal combinations of polynomial degree $s$ and number $r$ of quadrature points is $r=s$ and $r=s+1$, respectively. That is, the integrators $(PsNsQ2sGau)$
and $(PsNs+1Q2sLob)$ are the most efficient integrators with view of run-time per
order.
%
%

For a clearer illustration we omit in Fig.~\ref{fig:Laufzeit_Fehle die besten}
the less efficient integrators which are displayed in Fig.~\ref{fig:Laufzeit_Fehler} and include the most efficient integrators
constructed with the Gauss and Lobatto quadrature. 
It can be read off which of these
integrators provide the desired accuracy with shortest
run-time. Larger $s$  means longer processing time but also higher
order. Notice that $(P1N1Q2Gau)$ and $(P1N2Q2Lob)$ are the midpoint
rule and the St\"ormer-Verlet method, respectively.

\subsubsection{Conservation of angular momentum and long-time energy behavior}

Positioning the system of the two-dimensional harmonic oscillator
in the $(x,y)$-plane of the three-dimensional space, the $z$-component
of the angular momentum 
\[
I(p,q)=-p_{1}q_{2}+p_{2}q_{1}
\]
is a conserved quantity. This follows from
Noether's theorem because the Lagrangian of the two-dimensional harmonic
oscillator $L=\frac{1}{2}\dot{q}^{T}\dot{q}-\frac{1}{2}q^{T}q$ is
invariant under the group of rotations $SO(2)=\{B\in\mathbb{R}^{2\times2}\,|\, B^{T}B=\text{Id},\,\,\det(B)=1\}$.
With the linearity of $R_{v} \in SO(2)$ and Remark~\ref{rem:lingroup} it follows that all variational integrators
$(PsNrQu)$ conserve the $z$-component of the angular momentum. In Fig.~\ref{fig:Drehimpuls-Fehler} it is shown that the $z$-component of the
angular momentum, if the variational integrators $(P2N3Q4Lob)$,
$(P3N4Q6Lob)$ and $(P4N5Q8Lob)$ are used, is preserved up to an error less than
$10^{-14}$.\footnote{Note that the accuracy is limited to machine precision and the accuracy of the applied Newton method to solve the discrete Euler-Lagrange equations.} 
In Fig.~\ref{fig:Drehimpuls-Fehler} (right)
the behavior of a Runge-Kutta method of order four is included which is not symplectic nor momentum-preserving.
Thus, Noether's theorem does not apply and the angular momentum is not preserved.
\begin{figure}[h]
\centering
\includegraphics[clip,width=0.5\textwidth]{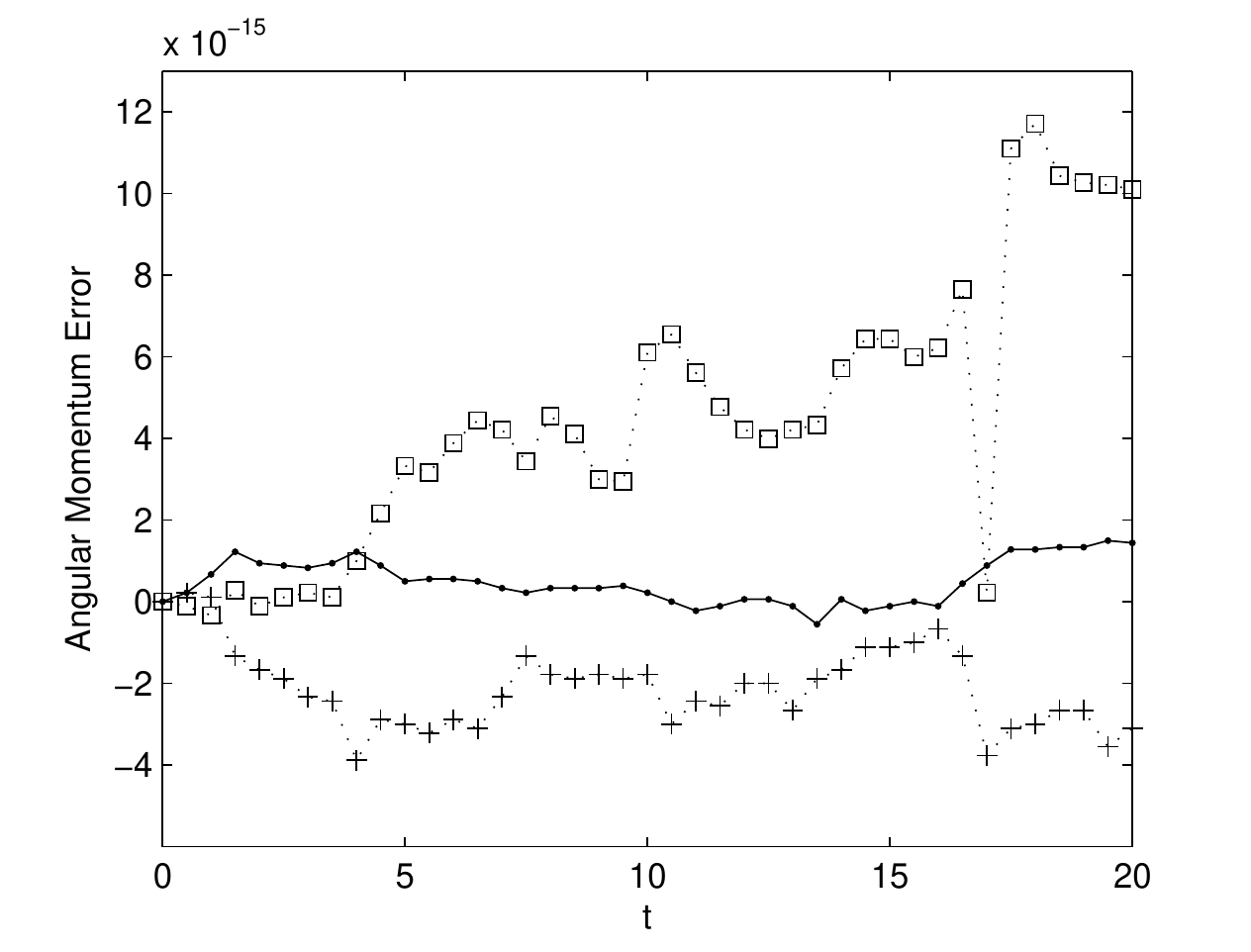}\includegraphics[clip,width=0.5\textwidth]{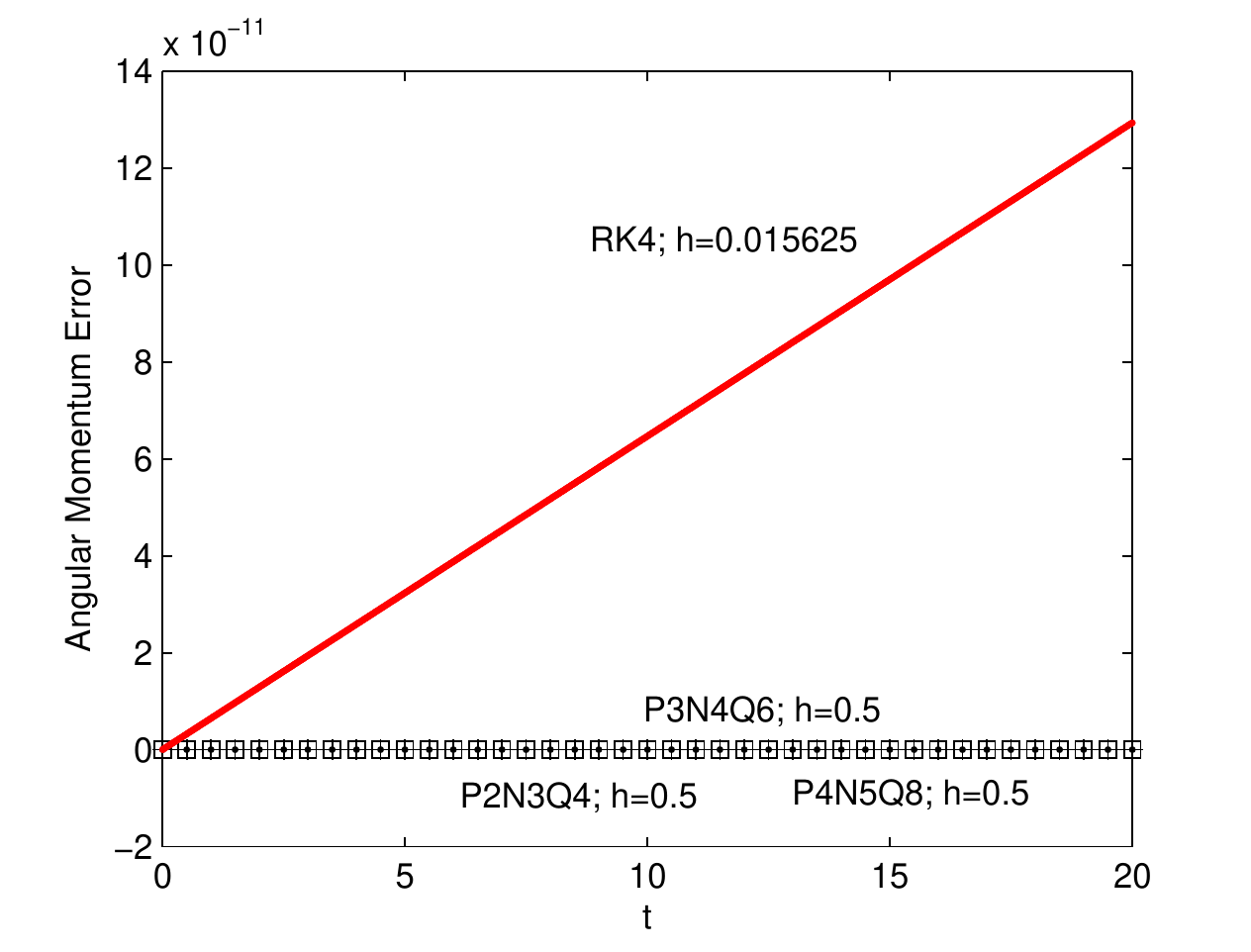}
\caption{Harmonic oscillator:
Left: Error of the $z$-component of the angular momentum for the variational
integrators $\left(P2N3Q4Lob\right)$ (dots), $\left(P3N4Q6Lob\right)$
(crosses) and $\left(P4N5Q8Lob\right)$ (squares) with step size $h=0.5$.
Right: Same plot as on the left with non-variational integrator included (Runge-Kutta
method of order four with step size $h=2^{-6}$ (red solid)) .}\label{fig:Drehimpuls-Fehler}
\end{figure}
In Section~\ref{subsec:prop} we introduced
the good long-time energy behavior of symplectic integrators,
in particular of variational integrators. In Fig.~\ref{fig:EnergieFehler} the error of the total energy of the harmonic oscillator simulated by different integrators is shown.
While the use of the variational integrators $\left(P3N4Q6Lob\right)$ and $\left(P4N5Q8Lob\right)$ leads to an oscillating but stable energy behavior, the use of a nonsymplectic Runge-Kutta method of order
four clearly shows an energy drift.
\begin{figure}[h]
\centering
\includegraphics[width=\textwidth]{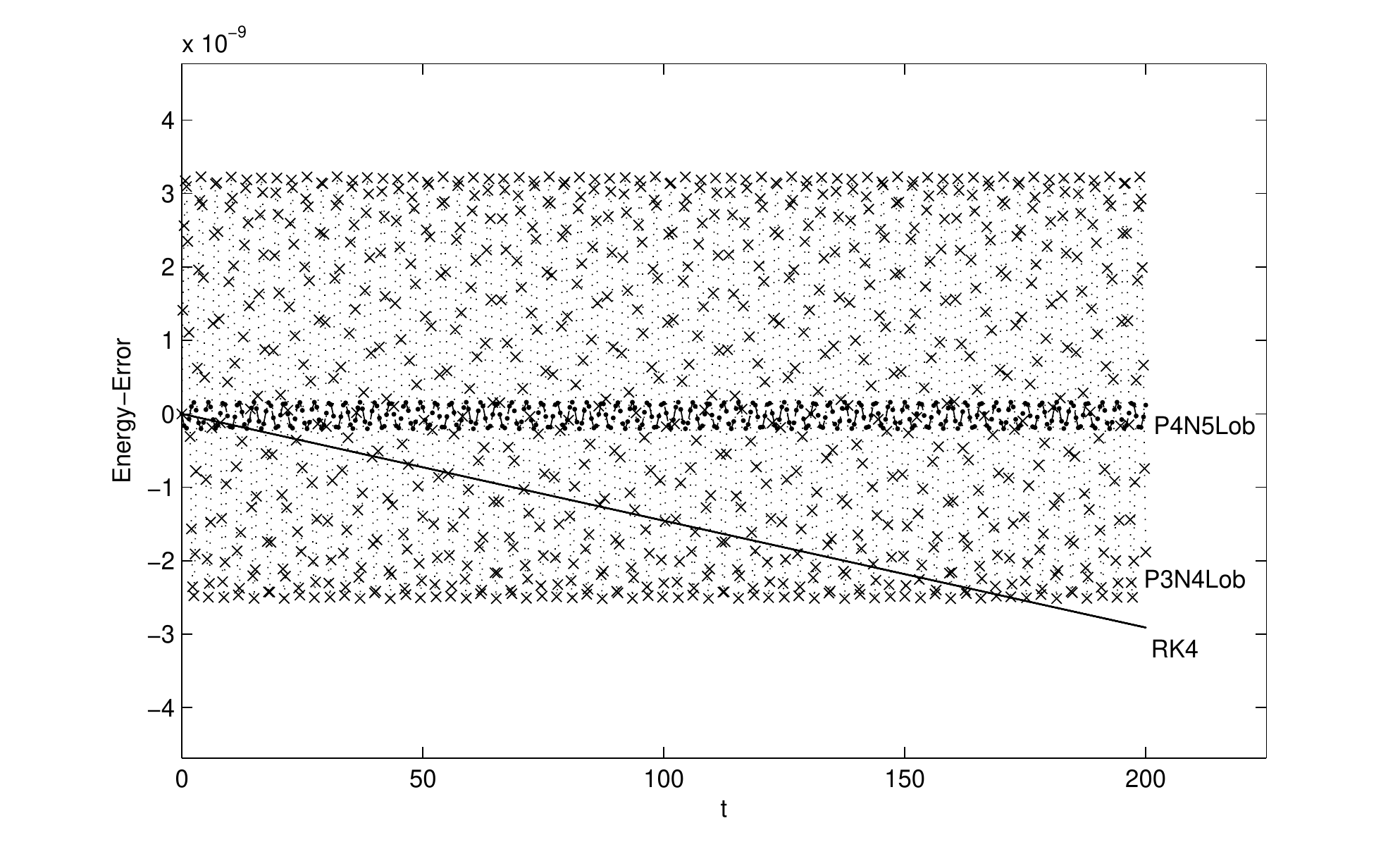}
\caption{Harmonic oscillator: Error
of the energy for the integrators $\left(P3N4Q6Lob\right)$ with step
size $h=0.25$ (crosses), $\left(P4N5Q8Lob\right)$ with step size $h=0.5$
(dots) and for a Runge-Kutta method of order four with step size $h=2^{-6}$
(solid line).}\label{fig:EnergieFehler}
\end{figure}

\subsection{Kepler problem}\label{subsec:kepler}
The two-body problem, also known as Kepler problem, is to determine the motion
of two point particles with masses $m_{1}, m_2 \in\mathbb{R}$. By assuming $m_1=1$ with
gravitational constant $\gamma$ and $k=\gamma m_{1}m_{2}\in\mathbb{R}$
we construct the Lagrangian of the Kepler problem 
\begin{equation}\label{eq:lagK}
L(q,\dot{q})=\frac{1}{2}\dot{q}^{T}\dot{q}+\frac{k}{\sqrt{q_{1}^{2}+q_{2}^{2}}}
\end{equation}
and the Hamiltonian 
\[
H(q,p)=\frac{1}{2}p^{T}p-\frac{k}{\sqrt{q_{1}^{2}+q_{2}^{2}}}
\]
with $q=(q_{1},q_{2})^{T},\dot{q},p\in\mathbb{R}^{2}$. The Hamiltonian
equations provide the system of differential equations which we want
to solve with respect to the initial condition $(q_{0},p_{0})$. For the simulations we set $k=1.016895192894334 \cdot 10^3$, $q_{0}=(5,0)^{T}$ and $p_{0}=(0,17)^{T}$
since this results in a motion where the mass $m_{1}$ describes
an elliptic orbit around mass $m_{2}$ with period $T=5.0$. Therefore,
after $\frac{25}{h}$ steps of step size $h$ the simulated mass should
return the fifth time to the initial point $(q_{0},p_{0})$. 

\subsubsection{Numerical convergence order}

As error we compute the maximal difference between the given initial value and the value after $N=\frac{25}{h}$ steps of integration
given as
\[
\max_{i\in\{1,2\}}|q_{N,i}^{0}-q_{0,i}| \quad\text{and}\quad \max_{i\in\{1,2\}}|p_{N,i}^{0}-p_{0,i}|
\]
for configuration and momentum, respectively.
The error is shown in Fig.~\ref{fig:Gaukonvplot-1} and Fig.~\ref{fig:Lobkonvplot-1}
for the different variational integrators $(PsNrQu)$ and for step
sizes $h\in\{1,0.5,0.25,0.125,0.1,0.0625,0.03125\}$. The
numerically determined order is given in Fig.~\ref{fig:Ordnung-der-variatinellen-2}
and Fig.~\ref{fig:Ordnung-der-variatinellen-1-1} and coincides nicely with the
orders for the harmonic oscillator as given in
Tables~\ref{tab:Gau} and \ref{tab:Lob}.
It can be observed in
the plots that the error of the integrators decreases not far below $10^{-10}$.
Since the iteration is implicit and has to be solved by a Newton method, the accuracy is limited by the machine precision and the used solver (we used \texttt{fsolve} implemented in Matlab).
\begin{figure}[h]
\centering
\includegraphics[width=\textwidth]{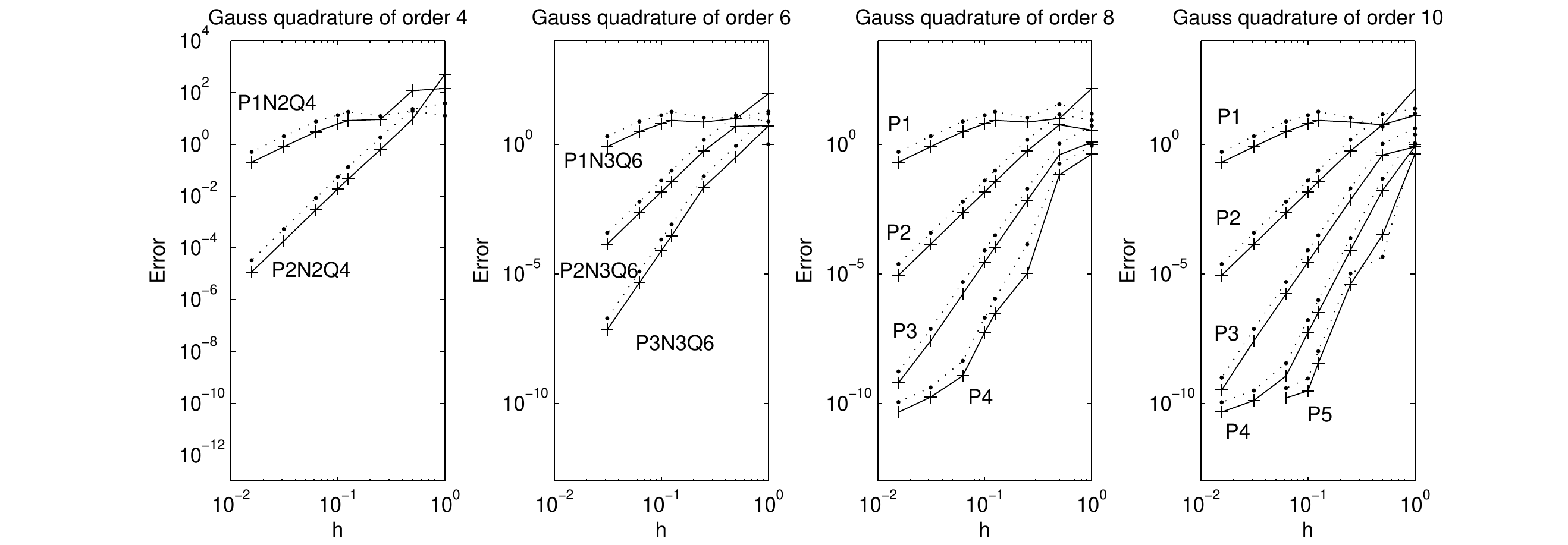}
\caption{\label{fig:Gaukonvplot-1}{Kepler problem: Log-log plot
of the error for the configurations $q$ (crosses) and momenta $p$
(dots) with step size $h$ for the integrators $(PsNrQ2rGau)$; divided
into four subplots, separated by applied Gauss quadrature. In the
third subplot $Ps$ denotes $(PsN4Q8Gau)$ and in the fourth $Ps$
denotes $(PsN5Q10Gau)$.}}
\end{figure}
\begin{figure}[h]
\centering
\includegraphics[width=\textwidth]{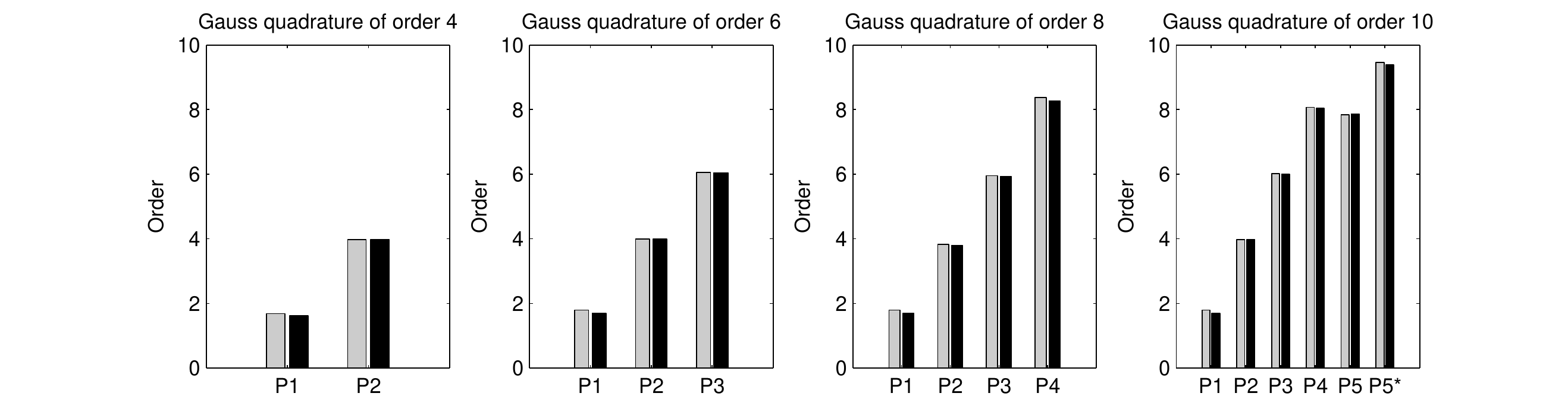}
\caption{\label{fig:Ordnung-der-variatinellen-2}{Kepler problem:
Order of the variational integrators $\left(PsNrQ2rGau\right)$ with
respect to the configurations $q$ (left bar) and the momenta $p$
(right bar). The title gives the order of the applied Gauss quadrature
which is $2r$. The $x$-axis indicates the degree of the used polynomial.
$P5$ gives the numerically determined order with use of the first six values
whereas $P5^{*}$ gives the numerically determined order with use of the first
five values.}}
\end{figure}
\begin{figure}[h]
\centering
\includegraphics[width=\textwidth]{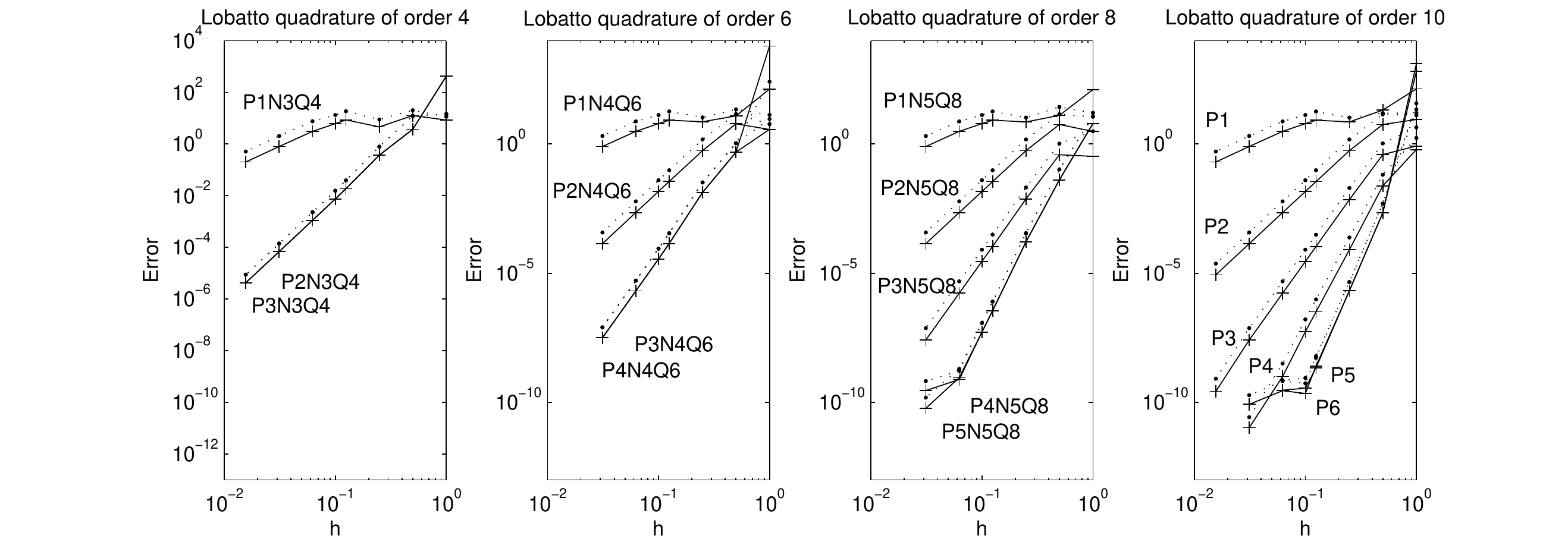}
\caption{\label{fig:Lobkonvplot-1}{Kepler problem: Log-log plot
of the error for the configurations $q$ (crosses) and momenta $p$
(dots) with step size $h$ for the integrators $(PsNrQ2r-2Lob)$;
divided into four subplots, separated by applied Lobatto quadrature.
The values of $(Ps-1NsQ2s-2Lob)$ and $(PsNsQ2s-2Lob)$ are superimposed,
they are of the same order. In the last subplot $Ps$ denotes $(PsN6Q10Lob)$.}}
\end{figure}
\begin{figure}[h]
\centering
\includegraphics[width=\textwidth]{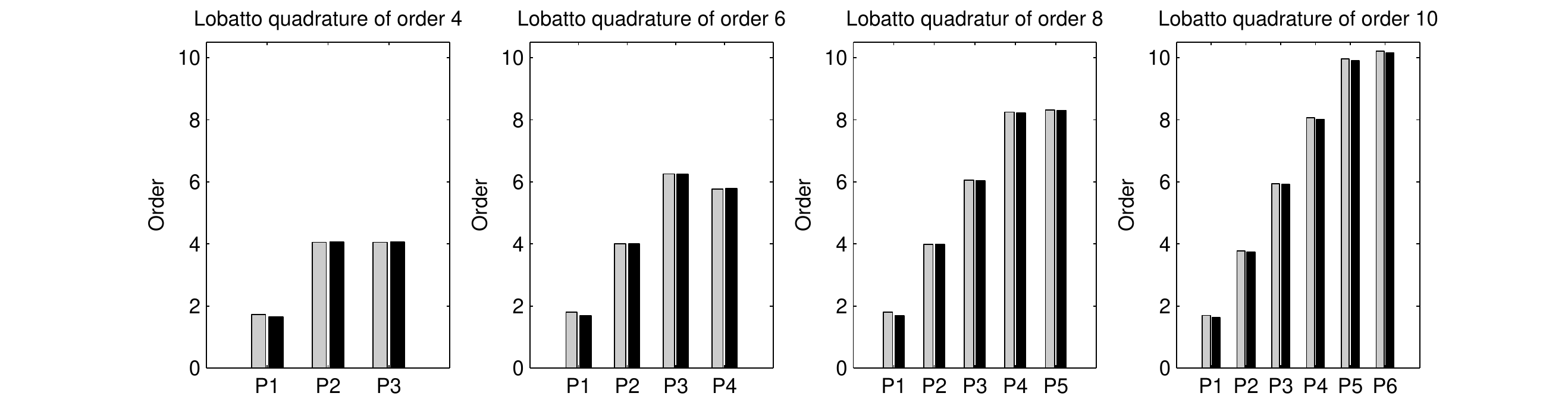}
\caption{\label{fig:Ordnung-der-variatinellen-1-1}{Kepler problem:
Order of the variational integrators$\left(PsNrQ2r-2Lob\right)$ with
respect to the configurations $q$ (left bar) and the momenta $p$
(right bar). The title gives the order of the applied Lobatto quadrature
which is $2r-2$. The $x$-axis indicates the degree of the used polynomial.}}
\end{figure}

\subsubsection{Conservation of angular momentum and long-time energy behavior}

Since the Lagrangian~\eqref{eq:lagK} of the Kepler problem
is invariant under the group of rotations $SO(2)=\{B\in\mathbb{R}^{2\times2}\,|\, B^{T}B=\text{Id},\,\,\det(B)=1\}$,
the angular momentum
\[
I(p,q)=-p_{1}q_{2}+p_{2}q_{1}
\]
is a conserved quantity in the system. The angular momentum is (up to numerical accuracy) also preserved in the discrete solution using the variational integrators
$(PsNrQu)$ (cf.~Remark~\ref{rem:lingroup}) as shown in Fig.~\ref{fig:Drehimpuls-Fehler-1} for the integrators $\left(P2N2Q4Gau\right)$,
$\left(P3N3Q6Gau\right)$ and $\left(P4N4Q8Gau\right)$. 
However, using a Runge-Kutta integrator of order four, angular momentum is not preserved anymore as shown in Fig.~\ref{fig:Drehimpuls-Fehler-1} (right).
\begin{figure}[h]
\centering
\includegraphics[clip,width=0.5\textwidth]{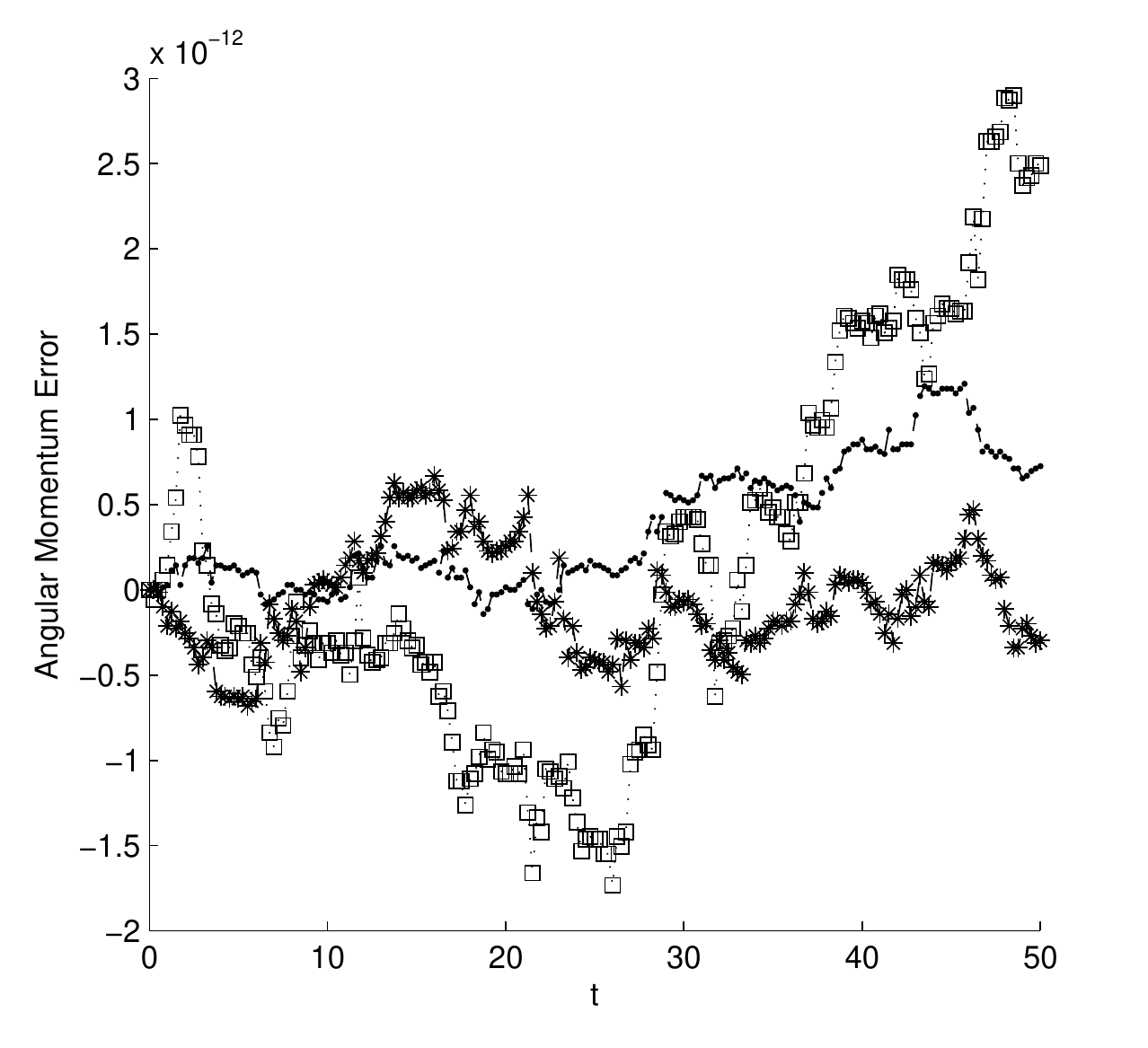}\includegraphics[clip,width=0.5\textwidth]{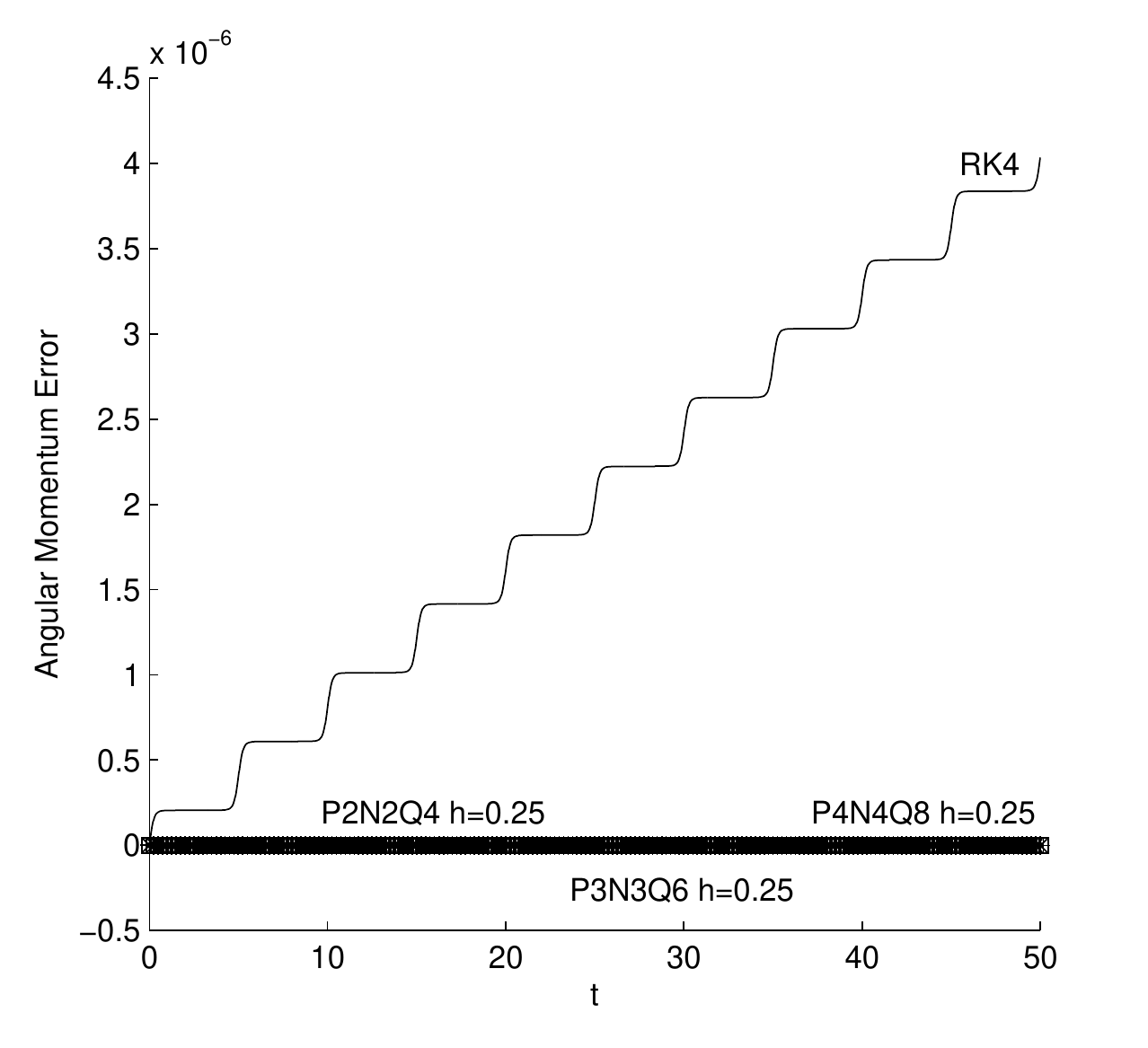}
\caption{Kepler problem: Left:
Error of the angular momentum for the variational
integrators $\left(P2N2Gau\right)$ (dots), $\left(P3N3Gau\right)$
(stars) and $\left(P4N4Gau\right)$ (squares) with
step size $h=0,25$. Right: Same plot as on the left with non-variational
integrator added (Runge-Kutta method of order four with step size
$h=2^{-6}$ (red solid)). }\label{fig:Drehimpuls-Fehler-1}
\end{figure}
The error in the energy is given in Fig.~\ref{fig:EnergieFehler-1}.
As for the harmonic oscillator, due to the symplecticity of the variational integrators the error of the integrators $\left(P3N3Gau\right)$ and $\left(P4N4Gau\right)$ is oscillating but bounded whereas the Runge-Kutta solution exhibits an energy drift.
\begin{figure}[h!]
\centering
\includegraphics[width=\textwidth]{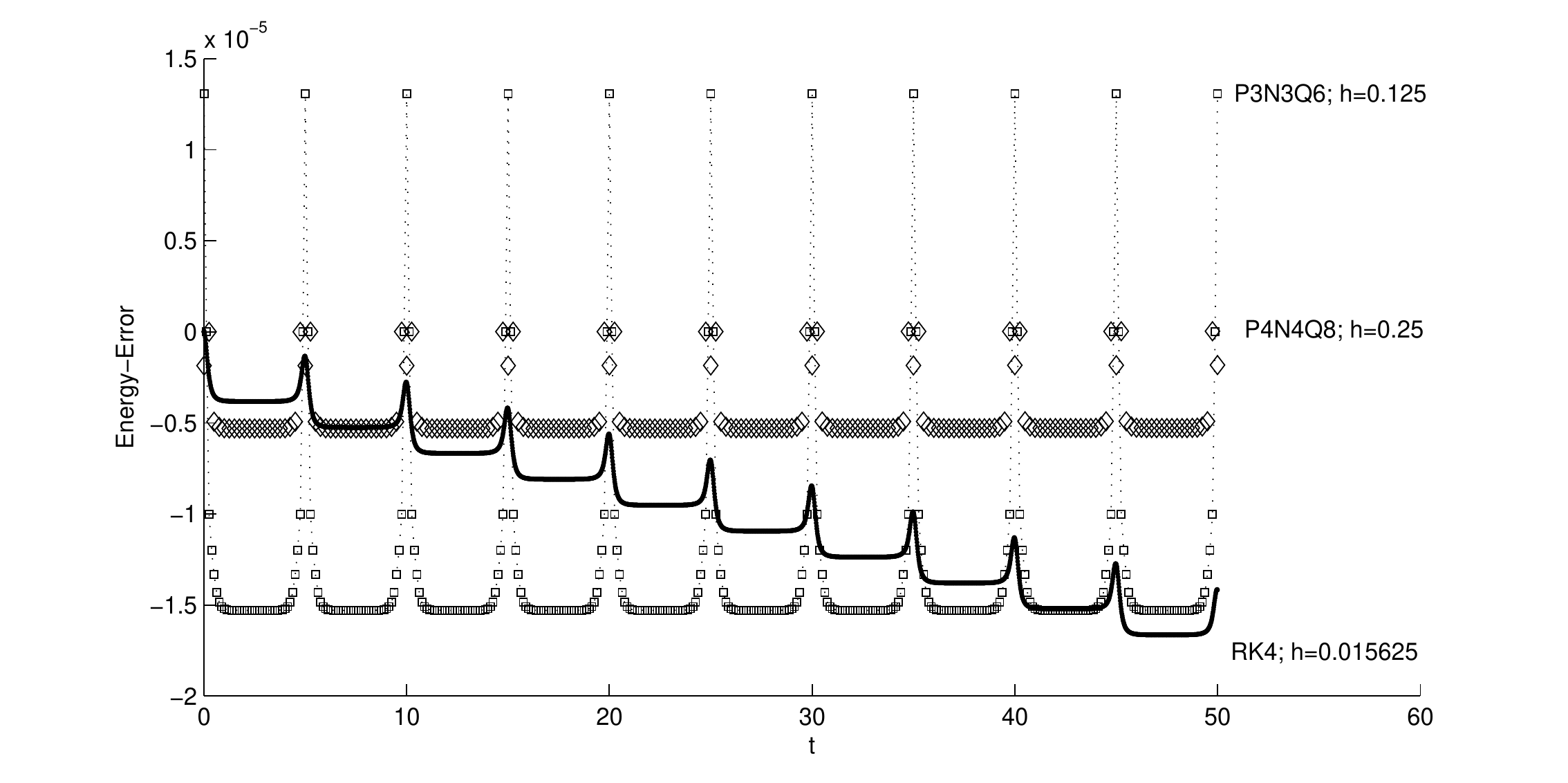}
\caption{Kepler problem: Error of
the energy for $\left(P3N3Q6Gau\right)$ with step size $h=0.125$
(squares), $\left(P4N4Q8Gau\right)$ with step size $h=0.25$ (diamonds)
and for a Runge-Kutta method of order four with step size $h=2^{-6}$
(solid line).}\label{fig:EnergieFehler-1}
\end{figure}
\begin{figure}[h]
\centering
\includegraphics[angle=90,width=13cm]{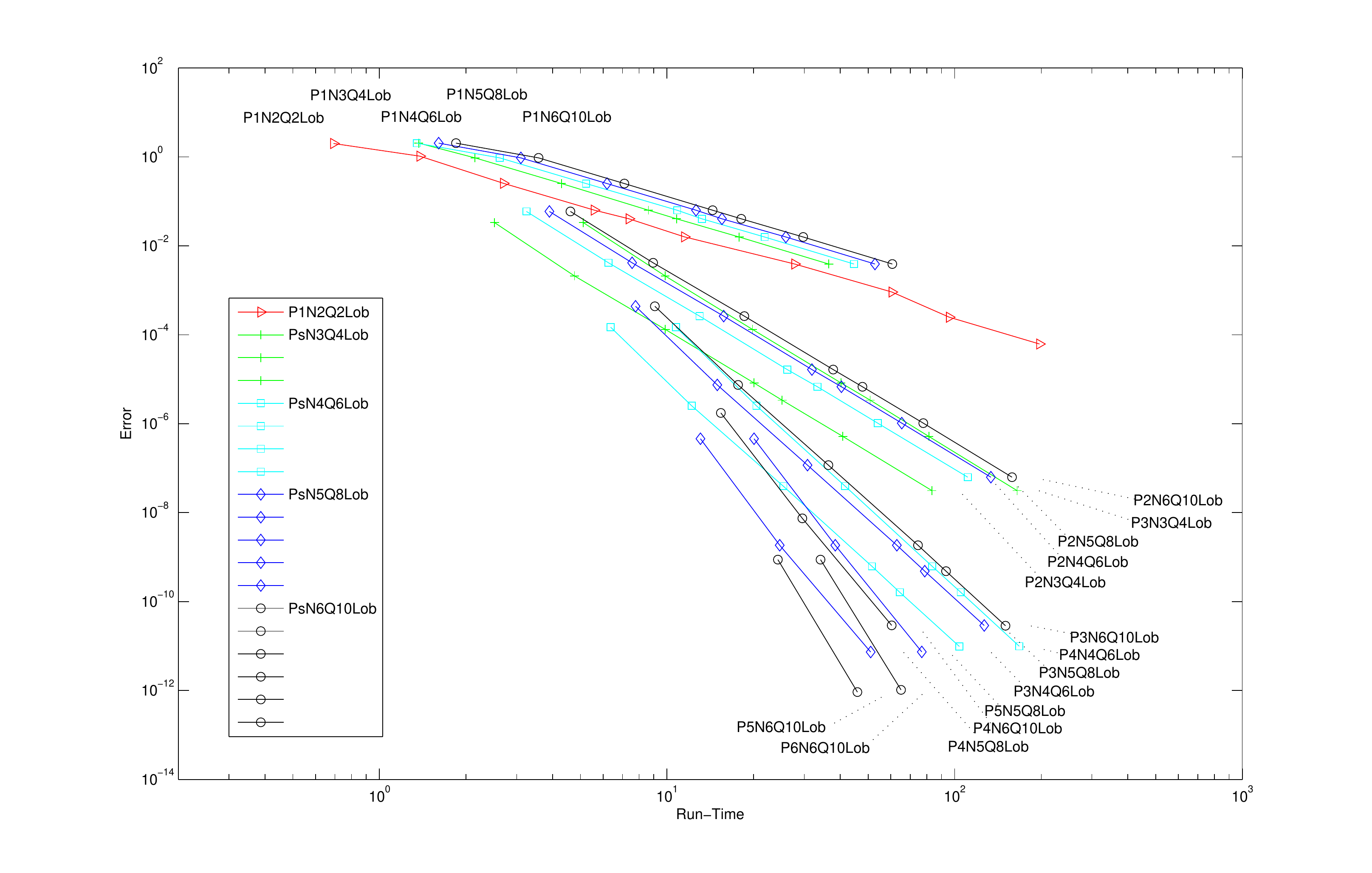}
\caption{Harmonic oscillator: Global error
against run-time for the integrators $(PsNrQ2r-2Lob)$ and different
step sizes $h$.}\label{fig:Laufzeit_Fehler}
\end{figure}
\begin{figure}[h]
\centering
\includegraphics[angle=90,width=13cm]{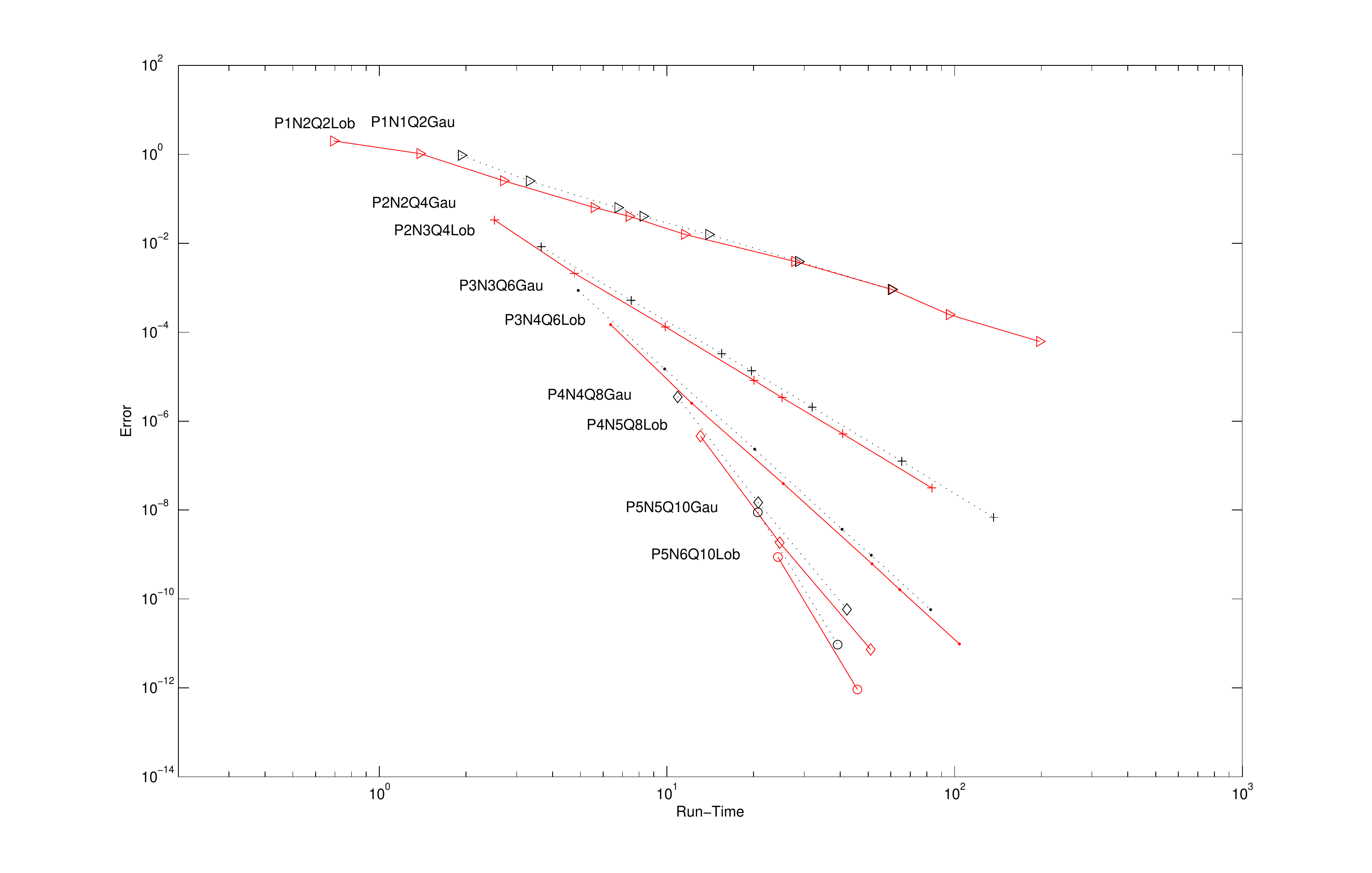}
\caption{Harmonic oscillator:
Global error against run-time for the integrators
$(PsNsQ2sGau)$ (dotted line) and $(PsNs+1Q2sLob)$ (solid line)
for $s=\{1,...,6\}$ and different step sizes $h$.}\label{fig:Laufzeit_Fehle die besten}
\end{figure}

\section{Conclusion}\label{sec:concl}

In this work variational integrators of higher order have been constructed by following the approach of Galerkin variational integrators introduced in \cite{MaWe01}.
Thereby, the solution of the Euler-Lagrange equations is approximated by a polynomial of degree $s$ and the action by a quadrature formula based on $r$ quadrature points. The restriction to $r=s$ quadrature points (as assumed in \cite{MaWe01}), which leads to symplectic partitioned Runge-Kutta methods, is dropped.
For the resulting methods the order of convergence has been determined numerically for two numerical examples. It has been numerically demonstrated that the order of convergence can be increased by adapting the number of quadrature points to the polynomial degree. In particular, if the Lobatto quadrature is used, the choice of $r=s+1$ leads to an integrator of order $2s$ which is of two orders higher compared to an integrator with $r=s$ quadrature points (order $2s-2$).
Thus, the ideal ratio between the number of quadrature points $r$ and the polynomial degree $s$ could be determined for different quadrature rules and the numerical results predict that $2s$ is the maximal possible order of the constructed variational integrators.
The structure-preserving properties such as symplecticity, momentum-preservation and good long-time energy behavior have been demonstrated by numerical examples.
In addition, for symmetrically distributed control points of the polynomial, the variational integrators $(PsNrQ2rGau)$ and $(PsNrQ2r-2Lob)$ have been shown to be time-reversible.
Furthermore, a linear stability analysis has been performed for selected integrators and stability regions have been determined.
It has been shown that the integrators $(PsNsQ2sGau)$ are A-stable, i.e.~there are no stability restrictions on the step size $h$. 
 
In the future, a formal proof of the numerically determined convergence order of $\min{(2s,u)}$ has to be performed. To this end, techniques of variational error analysis \cite{MaWe01,HaLe13} can be applied which are based on the following main idea: Rather than considering how closely the trajectory of the discrete Hamiltonian map matches the exact trajectory, one considers how closely the discrete Lagrangian matches the action integral. It is shown in \cite{MaWe01} and \cite{PatCu} that both order concepts are equivalent. The analytical computation of the variational error defined in this way for the general Galerkin variational integrators introduced in this work is still subject of ongoing research.
The application of the constructed higher order variational integrators to holonomic and nonholonomic integrators is straightforward, however, a careful analysis has to be carried out to see if the predicted orders also hold for these systems.
Currently, the approach is used for the optimal control of mechanical systems (\cite{COJ12}) 
and numerical results confirm that also the adjoint resulting from the necessary optimality condition inherits the same order of the variational scheme as it also does for symplectic partitioned Runge-Kutta methods (cf.~\cite{ObJuMa10}). 
Another topic of interest is the construction of time-adaptive variational integrators since naive time-adapting strategies destroy the structure-preserving properties (see e.g.~\cite{Leimkuhler2004}). The higher order integrators could be applied to adapt the order rather than to adapt the step size, e.g.~a higher order scheme can be deployed if higher accuracy requirements have to be matched.
Furthermore, for systems involving slow and fast time scales (see e.~g.~\cite{LO12}) variational integrators of different orders for the different subsystems can be used to increase the efficiency of the simulations. Thereby, an investigation regarding preserved quantities and long-time behavior is essential.

{\small
\bibliographystyle{alpha}
\bibliography{referencesAntrag}\clearpage
}

\end{document}